\documentclass[12pt, reqno]{amsart}
\usepackage{xr}
\usepackage{amsthm,amssymb,amstext,amscd,amsfonts,soul}
\usepackage{amsbsy,amsxtra,latexsym,url,multirow}
\usepackage{amsmath,color,comment}
\usepackage{fancybox}
\usepackage{fullpage}
\usepackage[english]{babel}
\usepackage[latin1]{inputenc}
\usepackage{float}\restylefloat{table}
\usepackage{bm}
\usepackage{enumitem}
\usepackage[draft]{hyperref}

\newcommand{\field}[1]{\mathbb{#1}}
\newcommand{\N}{\field{N}}
\newcommand{\Z}{\field{Z}}%
\newcommand{\R}{\field{R}}
\newcommand{\C}{\field{C}}
\newcommand{\Q}{\field{Q}}
\newcommand{\GL}{\operatorname{GL}}
\newcommand{\HH}{\mathbb{H}}
\newcommand{\sgn}{\mathrm{sgn}}
\newcommand{\SL}{\operatorname{SL}}
\newcommand{\sign}{sign}
\newcommand{\diag}{diag}
\newcommand{\Res}{\text{Res}}
\newcommand{\im}{\text{Im}}
\newcommand{\re}{\text{Re}}
\newcommand{\triplet}{\mathcal{W}(p)}
\newcommand{\singlet}{\mathcal{W}(2,2p-1)}
\def \a{\alpha}
\newcommand{\vol}{\operatorname{vol}}
\newcommand{\wtb}{{\rm {wt} }  b }
\newcommand{\bea}{\begin{eqnarray}}
\newcommand{\eea}{\end{eqnarray}}
\newcommand{\be}{\begin {equation}}
\newcommand{\ee}{\end{equation}}
\newcommand{\g}{\frak g}
\newcommand{\nn}{\nonumber \\}
\newcommand{\ver}{L(-\tfrac{4}{3}\Lambda_0)}
\def \ga{\gamma }
\newcommand{\hg}{\widehat {\frak g} }
\newcommand{\hn}{\widehat {\frak n} }
\newcommand{\h}{\frak h}
\newcommand{\wt}{{\rm {wt} }   }
\newcommand{\V}{\cal V}
\newcommand{\hh}{\widehat {\frak h} }
\newcommand{\logmin}{\mathcal{W}(p,p')}
\newcommand{\slogmin}{\mathcal{SW}(p,p')}

\newcommand{\n}{\frak n}
\newcommand{\W}{\mathcal W}
\newcommand{\vak}{\bf 1}

\newcommand{\la}{\langle}
\newcommand{\ra}{\rangle}
\newcommand{\ap}{{\bf a}^+}
\newcommand{\am}{{\bf a}^-}

\newcommand{\NS}{\frak{ns} }
\newcommand{\striplet}{\mathcal{SW}(m)}
\newcommand{\hf}{\mbox{$\frac{1}{2}$}}
\newcommand{\thf}{\mbox{$\frac{3}{2}$}}
\newcommand{\doublet}{\mathcal{A}(p)}
\renewcommand{\sl}{{\operatorname{s\ell}}}
\newcommand{\ord}{\operatorname{ord}}
\newcommand{\lcm}{\operatorname{lcm}}
\renewcommand{\H}{\mathbb{H}}

\numberwithin{equation}{section}
\numberwithin{table}{section}
\newtheorem{theorem}{\textbf{Theorem}}
\numberwithin{theorem}{section}
\newtheorem{lemma}[theorem]{\textbf{Lemma}}

\theoremstyle{remark}

\newtheorem{remarks}[theorem]{{\bf Remarks}}

\renewcommand{\pmod}[1]{\  \,  \left(  \mathrm{mod} \,  #1 \right)}

\begin{document}
\title{Class numbers and representations by ternary quadratic forms with congruence conditions}
\author{Kathrin Bringmann}
\address{Mathematical Institute, University of Cologne, Weyertal 86-90, 50931 Cologne, Germany}
\email{kbringma@math.uni-koeln.de}
\author{Ben Kane}
\address{Department of Mathematics, University of Hong Kong, Pokfulam, Hong Kong}
\email{bkane@hku.hk}
\date{\today}
\maketitle

\begin{abstract}
In this paper, we are interested in the interplay between integral ternary quadratic forms and class numbers. This is partially motivated by a question of Petersson. 	
\end{abstract}

\section{Introduction and statement of results}
For $\bm{a}\in\N^3$ and $\bm{x}\in\Z^3$, we define
\[
Q_{\bm{a}}(\bm{x}):=\sum_{j=1}^3 a_j x_j^2.
\]
Throughout we write vectors in bold letters and their components in non-bold and with subscripts. For $\bm{h}\in\Z^3$, $N\in\N$, and $n\in\N_0$, we then let 
\[
r_{\bm{a},\bm{h},N}(n):=\#\left\{\bm{x}\in\Z^3:Q_{\bm{a}}(\bm{x})=n\text{ and }\bm{x}\equiv \bm{h}\pmod{N}\right\}.
\]
We omit $\bm{h}$ and $N$ if $N=1$. 

It is well-known that $r_{\bm{1}}(n)$ (with $\bm{1}:=(1\ 1\ 1)^T$) may be written in terms of {\it Hurwitz class numbers} $H(|D|)$ ($D<0$ a discriminant) which count the number of $\SL_2(\Z)$-equivalence classes of integral binary quadratic forms of discriminant $D$, weighted
 by $\tfrac{1}{2}$ if the form is equivalent to $f(x_1^2+x_2^2)$ and by $\tfrac{1}{3}$ if the form is equivalent to $f(x_1^2-x_1x_2+x_2^2)$,  for some $f\in\Z$. By Gauss (see e.g. \cite[Theorem 8.5]{OnoBook}) $r_{\bm{1}}(n)$ is proportional to class numbers 
\begin{equation}\label{eqn:r1Gauss}
r_{\bm{1}}(n)=\begin{cases}
12 H(4n)&\text{if }n\equiv 1,2\pmod{4},\\
24 H(n)&\text{if }n\equiv 3\pmod{8},\\
r_{\bm{1}}\left(\tfrac{n}{4}\right)&\text{if }4\mid n,\\
0&\text{otherwise}.
\end{cases}
\end{equation}
In this paper, we extend \eqref{eqn:r1Gauss} to include congruence conditions on $\bm{x}$ and to write such identities more uniformly. This is partially motivated by a question of Petersson \cite{PeBer}, who observed that for many choices of $\bm{h}$, the numbers $r_{\bm{1},\bm{h},4}(n)$ satisfy formulas reminiscent of \eqref{eqn:r1Gauss}. His claims for $\bm{h}=(0\ 0\ 1)^T$ and $\bm{h}=(1\ 2\ 2)^T$ were later proved by Ebel \cite{Ebel}. To give a flavor of the types of identities that one obtains, for $\bm{h}=(0\ 1\ 1)^T$ we have 
\begin{equation}\label{eqn:Ratio1}
r_{\bm{1},\bm{h},4}(n)=\tfrac{1}{12}\delta_{n\equiv 2\pmod8} r_{\bm{1}}(n)=\delta_{n\equiv 2\pmod8}H(4n),
\end{equation}
where here and throughout $\delta_S:=1$ if a statement $S$ holds and $\delta_S:=0$ otherwise.
Using \eqref{eqn:r1Gauss} and the first identity in \eqref{eqn:Ratio1} as a guide, we undertake a two-part search; the first search yields choices of $\bm{a}$ for which identities like \eqref{eqn:r1Gauss} hold, while the second search yields choices of $\bm{h}$ and $N$ for which an identity resembling the first identity in \eqref{eqn:Ratio1} holds. 

Prior work of Jones (see \cite[Theorem 86]{Jones}) immediately yields the choices of $\bm{a}$, namely  those for which $Q_{\bm{a}}$ has \begin{it}(genus) class number one\end{it}, which means that every integral quadratic form which is $p$-adically equivalent to $Q_{\bm{a}}$ is actually equivalent over $\Z$. Jones showed that, while the numbers $r_{\bm{a},\bm{h},N}(n)$ may not themselves be directly related to class numbers, a certain weighted average (known as the Siegel-Weil average) $r_{\operatorname{g}(Q_{\bm{a}})}(n)$ may indeed be written in terms of class numbers. If $Q_{\bm{a}}$ has class number one, then the average collapses to 
\[
r_{\operatorname{g}(Q_{\bm{a}})}(n)=r_{\bm{a}}(n).
\]
The (finite set of) binary integral quadratic forms with class number one is known and we let $\mathcal{C}$ denote the set of $\bm{a}$ for which $Q_{\bm{a}}$ has class number one. 
For each $\bm{a}\in\mathcal{C}$, we obtain a formula relating $r_{\bm{a}}(n)$ to Hurwitz class numbers. For example, Lemma \ref{lem:111} states that
\begin{equation}\label{eqn:r1notGauss}
r_{\bm{1}}(n)=12H(4n)-24H(n).
\end{equation}
For each $\bm{a}\in\mathcal{C}$, we find a set $\mathcal{S}_{\bm{a}}$ of  $(\bm{h},N)$ for which an identity resembling \eqref{eqn:Ratio1} holds. For each $(\bm{h},N)\in\mathcal{S}_{\bm{a}}$, we show that $r_{\bm{a},\bm{h},N}(n)$ is proportional to $r_{\bm{a}}(n)$; the constant of proportionality $d_{\bm{a},\bm{h},N}(n)$ only depends on the residue class of $n\pmod{M}$ for some fixed $M$, mimicking \eqref{eqn:Ratio1}. We list $d_{\bm{a},\bm{h},N}(n)$ in Appendix \ref{sec:appendixds}; note that throughout Appendices \ref{sec:appendixds} and \ref{sec:appendixcs}, the constant is taken to be zero if a given congruence class does not occur in the tables. 
\begin{theorem}\label{thm:congcondition}
For each $\bm{a}\in \mathcal{C}$, $(\bm{h},N)\in\mathcal{S}_{\bm{a}}$, and $n\in\N$, we have 
\[
r_{\bm{a},\bm{h},N}(n)=d_{\bm{a},\bm{h},N}(n) r_{\bm{a}}(n).
\]
\end{theorem}
\begin{remarks}
\noindent

\noindent
\begin{enumerate}[leftmargin=*]
\item
There is also a definition of genus for a quadratic form $Q$ for which $\bm{x}$ is in a fixed congruence class $\bm{h} \pmod{N}$ and van der Blij \cite{vanderBlij} constructed a Siegel-Weil average $r_{\operatorname{g}(Q,\bm{h},N)}(n)$ for such quadratic forms with congruence conditions. 
For each prime $p$ (including the prime $\infty$), there are so-called local densities $\beta_{(Q,\bm{h},N),p}(n)$ for which 
\begin{equation}\label{eqn:rlocaldensity}
r_{\operatorname{g}(Q,\bm{h},N)}(n)=\beta_{(Q,\bm{h},N),\infty}(n)\prod_{p}\beta_{(Q,\bm{h},N),p}(n),
\end{equation}
 the product running over all finite primes $p$. For all but finitely many $p$, one has
\begin{equation}\label{eqn:localdensity}
\beta_{(Q,\bm{h},N),p}(n)=\beta_{Q,p}(n). 
\end{equation}
From this, one can conclude that $r_{\operatorname{g}(Q,\bm{h},N)}(n)$ and $r_{\operatorname{g}(Q)}(n)$ are proportional; the constant of proportionality depends on $p$-adic properties of $n$; $p$ not satisfying \eqref{eqn:localdensity}. Hence, based on \cite[Theorem 86]{Jones}, identities relating $r_{\bm{a},\bm{h},N}(n)$ to class numbers should only come from identities like the ones given in Theorem \ref{thm:congcondition} in the special case that $(Q_{\bm{a}},\bm{h},N)$ has class number one; by \cite[Corollary 4.4]{ChanOh}, this implies that $Q_{\bm{a}}$ has class number one. 
\rm

\item
In many cases, the identities given in Theorem \ref{thm:congcondition} may be deduced by noting that if $n$ is restricted to certain congruence classes, $Q_{\bm{a}}(\bm{x})=n$ implies that $\bm{x}\equiv \bm{h}\pmod{N}$ up to reordering of the variables. However, this is not always the case for the identities proven in Theorem \ref{thm:congcondition}. For example, combining Lemma \ref{lem:111} with Theorem \ref{thm:congcondition} (see the table in Appendix \ref{sec:appendixds}), we obtain that for $\bm{a}=(1\ 2\ 2)^T$ and  $\bm{h}\in\{(1\ 0\ 3)^T,(3\ 1\ 2)^T\}$ we have
\[
r_{{\bm{a}},\bm{h},6}(n)=\delta_{n\equiv 19\pmod{24}}\left(H(4n)-2H(n)\right).
\]
Hence there are at least two distinct possibilities for the solutions $\bm{x}\pmod 6$. It would be interesting to try to directly prove that these are the same via a combinatorial argument.

\end{enumerate}
\end{remarks}

Another motivation for finding relations between class numbers and representations by ternary quadratic forms comes from computational number theory. Specifically, the algorithm of Jacobson and Mosunov \cite{JacobsonMosunov} used to 
tabulate class groups of imaginary quadratic fields of all discriminants $|D|<2^{40}$ (setting a record for unconditional computation) exploits \eqref{eqn:r1Gauss} to evaluate  $H(n)$ or $H(4n)$ for squarefree $n\not\equiv 7\pmod{8}$. 
However, since $r_{\bm{1}}(n)=0$ for $n\equiv 7\pmod{8}$, \cite{JacobsonMosunov} required another method to compute $H(n)$ for squarefree $n\equiv 7\pmod{8}$. One may use the identities resembling \eqref{eqn:r1notGauss} for other choices of $\bm{a}$ to push the calculation further. Indeed, for $\bm{a}=(1\ 3\ 3)^T$, letting $\left(\tfrac{\cdot}{\cdot}\right)$ denote the extended Legendre symbol, we may show, using the methods of this paper that for $n\equiv 7\pmod{8}$ with $9\nmid n$, we have
\[
r_{\bm{a}}(n)=8\left(1+\left(\tfrac{n}{3}\right)\right)H(n).
\]
 Similarly, for $\bm{a}=(2\ 5\ 10)^T$, for $n\equiv 23\pmod{24}$ with $25\nmid n$, we have
\[
r_{\bm{a}}(n)=2\left(1-\left(\tfrac{n}{5}\right)\right) H(n).
\]

One may also use Theorem \ref{thm:congcondition} and the lemmas in Section \ref{sec:lemmas} to compute $H(n)$ quicker for special choices of $n$: since the shifted lattice defined by $\bm{x}\in\Z^3$ with $\bm{x}\equiv \bm{h}\pmod{N}$ is more ``spread out'' than the lattice $\Z^3$, the computation of $r_{\bm{a},\bm{h},N}(n)$ is faster than that of $r_{\bm{a}}(n)$. 

Although we use \cite[Theorem 86]{Jones} as a guide for finding identities like \eqref{eqn:r1notGauss} in Section \ref{sec:lemmas}, we do not employ it directly. Jones used a constructive proof to establish a bijection between (primitive) representations of integers and class numbers, but his theorem only applies to those $n$ which are coprime to $2a_1a_2a_3$. In principle, one could take \eqref{eqn:rlocaldensity} and compute the local densities on the right-hand side in order to determine $r_{\bm{a},\bm{h},N}(n)$ if the class number is one. One could then realize a relationship between the right-hand side as a special value of an $L$-function for a Dirichlet $L$-function and use Dirichlet's class number formula to obtain an identity involving class numbers. This is essentially what is done in \cite[Theorem 1.5]{ShimuraCongruence}, although many of the constants would still need to be worked out to obtain an explicit identity and the calculation splits depending on the $p$-adic properties of $n$. 

Instead, we  investigate the representation numbers $r_{\bm{a},\bm{h},N}(n)$ by packaging them into generating functions generally referred to as \begin{it}theta functions\end{it} ($q:=e^{2\pi i \tau}$)
\begin{equation}\label{3The}
\Theta_{\bm{a},\bm{h},N}(\tau)
:=\prod_{j=1}^3 \vartheta_{h_j,N}\left(2N a_j\tau\right)=\sum_{n\geq 0} r_{\bm{a},\bm{h},N}(n) q^n,\quad\text{where}
\end{equation}
\[
\vartheta_{h,N}(\tau):=\sum_{m\equiv h\pmod{N}} q^{\frac{m^2}{2N}}.
\]
The functions $\vartheta_{h,N}$ are modular forms of weight $\tfrac{3}{2}$ (see Lemma \ref{lem:thetamodular}), and we use the theory of non-holomorphic modular forms to recognize the generating functions of the right-hand sides of 
 \eqref{eqn:r1notGauss} as (holomorphic) modular forms of weight $\tfrac{3}{2}$. We then use the theory of modular forms to prove that these two modular forms are equal, and hence conclude that the coefficients used to build them are indeed equal. 

The paper is organized as follows. In Section \ref{sec:prelim}, we introduce and recall properties of non-holomorphic modular forms, theta functions, and class numbers. We then prove Theorem \ref{thm:congcondition} in Section \ref{sec:congcondition}. In Section \ref{sec:lemmas}, we confirm identities resembling \eqref{eqn:r1notGauss} for each $\bm{a}\in\mathcal{C}$, which, when combined with Theorem \ref{thm:congcondition}, yield the desired relations between class numbers representations by the form $Q_{\bm{a}}$ with $\bm{x}\equiv \bm{h}\pmod{N}$.  We list all of the constants needed for stating Theorem \ref{thm:congcondition} in Appendix \ref{sec:appendixds} and the constants required for the proof of Theorem \ref{thm:congcondition} are given in Appendix \ref{sec:appendixvalence}. Finally, Appendix \ref{sec:appendixcs} (resp. \ref{sec:appendixvalencelemmas}) gives the definitions of the constants needed to state (resp. prove) the lemmas in Section \ref{sec:lemmas}.

\section*{Acknowledgements}
The authors would like to thank Wai Kiu ``Billy'' Chan for many helpful comments. 
The research of the  first author is supported by the Alfried Krupp Prize for Young University Teachers of the Krupp foundation. The research of the second author was supported by grants from the Research Grants Council of the Hong Kong SAR, China (project numbers HKU 17316416, 17301317 and 17303618).
\section{Preliminaries}\label{sec:prelim}

\subsection{Modular forms}
We briefly introduce modular forms below, but refer the reader to \cite{OnoBook} for a more details. 
As usual, for $d$ odd, we set 
\[
\varepsilon_{d}:=\begin{cases} 1 &\text{if }d\equiv 1\pmod{4}\hspace{-1pt},\\ i&\text{if }d\equiv -1\pmod{4}\hspace{-1pt}.\end{cases}
\]
A function $f:\H\to\C$ satisfies \begin{it}modularity of weight $\kappa\in\frac12+\Z$ on $\Gamma\subseteq \Gamma_0(4)$\end{it} ($\Gamma$ a congruence subgroup containing $T:=\left(\begin{smallmatrix}1&1\\ 0 &1\end{smallmatrix}\right)$)
\textit{ with character $\chi$} if for every $\gamma=\left(\begin{smallmatrix}a&b\\ c&d\end{smallmatrix}\right)\in\Gamma$ we have 
\[
f|_{\kappa}\gamma  = \chi(d) f.
\]
Here the weight $\kappa$ \begin{it}slash operator\end{it} is defined by 
\[
f\big|_{\kappa}\gamma(\tau):= \left( \tfrac cd \right) \varepsilon_d^{2k}(c\tau+d)^{-\kappa} f(\gamma\tau).
\]
We moreover call $f$ a \begin{it}(holomorphic) modular form\end{it} if $f$ is holomorphic on $\H$ and $f(\tau)$ grows as most polynomially in $v$ as $\tau=u+iv\to \Q\cup\{i\infty\}$.

We call the equivalence classes of $\Gamma\backslash (\Q\cup\{i\infty\})$ the \begin{it}cusps of $\Gamma\backslash\H$\end{it}. For a cusp $\varrho$ we choose $M_{\varrho}$ such that $M_{\varrho}(i\infty)=\varrho$. If $f$ satisfies modularity of weight $\kappa$ on $\Gamma$ with some character $\chi$, then  $f_{\varrho}:=f\big|_{\kappa} M_{\varrho}$ is invariant under $T^{\sigma_{\varrho}}$ for 
some $\sigma_{\varrho}\in\N$ and
 hence has a Fourier expansion 
\[
f_{\varrho}(\tau)= \sum_{n\in\Z} c_{f,\varrho,v}(n) q^{\frac{n}{\sigma_{\varrho}}}.
\]
The number $\sigma_{\varrho}$ (chosen minimally) is called the \begin{it}cusp width\end{it} of $\varrho$, and we drop the dependence on $v$ in the notation if $f$ is holomorphic and we drop $\varrho$ from the notation if $\varrho=i\infty$. The growth conditions at the cusps for a holomorphic modular form are equivalent to the assumption that $c_{f,\varrho}(n)=0$ for every $n<0$. If $f$ is a holomorphic modular form and $m\in\Z$ is minimal such that $c_{f,\varrho}(m)\neq 0$, then we set 
\[
\ord_{\varrho}(f):=\tfrac{m}{\sigma_\varrho}.
\]

For $\tau\in \H$ and a congruence subgroup $\Gamma\subseteq\SL_2(\Z)$, we let 
\[
\omega_{\tau}=\omega_{\Gamma,\tau}:=\frac{\# \{\gamma\in\Gamma: \gamma \tau=\tau\}}{\#\{\gamma\in\Gamma : \gamma=\pm I\}}.
\]
The following identity, known as the valence formula, is used throughout this paper to prove identities between coefficients of modular forms.
\begin{lemma}\label{lem:valenceformula}
Suppose that $k\in\frac{1}{2}\Z$, $\Gamma\subseteq \Gamma_0(4)$ is a congruence subgroup containing $\left(\begin{smallmatrix}
1 & 1 \\ 0 & 1
\end{smallmatrix}\right)$, and $\chi$ is a character. If $f\not\equiv 0$ is a (holomorphic) modular form of weight $k$ on $\Gamma$ with character $\chi$, then 
\[
\tfrac{k}{12}\left[\SL_2(\Z):\Gamma\right]=\sum_{\tau\in \Gamma\backslash\H} \frac{\ord_{\tau}(f)}{\omega_{\tau}} +\sum_{\varrho\in \Gamma\backslash (\Q\cup\{i\infty\})} \ord_{\varrho}(f). 
\]
In particular, if $f(\tau)=\sum_{n\gg -\infty} c_{f}(n) q^n$ and $c_f(n)=0$ for all $n\leq \frac{k}{12}\left[\SL_2(\Z):\Gamma\right]$, then $f\equiv0$.
\end{lemma}

We require indices of certain subgroups of $\SL_2(\Z)$. 
As usual $\varphi$ denotes Euler totient function and we set $$\Gamma_{N,M}:=\Gamma_0(N)\cap\Gamma_1(M).$$

\begin{lemma}\label{lem:index}
If $N,M\in\N$ with $M\mid N$, then we have 
\[
\left[\SL_2(\Z):\Gamma_{N,M}\right]= N\prod_{p\mid N} \left(1+\tfrac{1}{p}\right) \varphi(M),
\]
where the product runs over all primes divisors of $N$.
In particular
\begin{equation*}
\left[\SL_2(\Z):\Gamma_0(N)\right]=N\prod_{p\mid N} \left(1+\tfrac{1}{p}\right).
\end{equation*}
\end{lemma}
\begin{proof}
We write 
\begin{equation*}
\left[\SL_2(\Z):\Gamma_{N,M}\right]=\left[\SL_2(\Z):\Gamma_0(N)\right]\left[\Gamma_0(N):\Gamma_{N,M}\right].
\end{equation*}
We have, by \cite[Proposition 1.7]{OnoBook},
\begin{equation*}
\left[\SL_2(\Z):\Gamma_0(N)\right]=N\prod_{p\mid N} \left(1+\tfrac{1}{p}\right).
\end{equation*}
Since $M\mid N$, there is a natural surjective homomorphism from $\Gamma_0(N)$ to $(\Z/M\Z)^{\times}$ with kernel $\Gamma_{N,M}$, from which we conclude that
\[
\left[\Gamma_0(N):\Gamma_{N,M}\right]=\varphi(M).\qedhere
\]
\end{proof}

\subsection{Operators on non-holomorphic modular forms}
For some of the theta functions, we directly prove that they match generating functions for certain class numbers by showing that the generating functions are modular forms and then use Lemma \ref{lem:valenceformula}. For this, we require some basic facts about operators on non-holomorphic modular forms. 
For $d , M \in \N$, $m \in \Z$, and $\chi$ a (real) character the operators on  
$f(\tau) = \sum_{n\in\Z}  c_{f,v}(n) q^n$ are defined as
\begin{align*}
f\big|U_d(\tau)&:=\sum_{n\in\Z} c_{f,\tfrac{v}{d}} \left(dn\right)q^n,
&f\big|V_d(\tau)&:=f(d\tau)=\sum_{n\in\Z} c_{f,dv}(n)q^{dn},\\
f\big|S_{M,m}(\tau)&:=\sum_{ n\equiv m\pmod{M}}c_{f,v}(n)q^n.& 
\end{align*}
The \begin{it}radical\end{it} for $n\in\N$ is
$
\operatorname{rad}(n):=\prod_{p\mid n}p.
$
For a discriminant $D$, let $\chi_D(\cdot):=\left(\tfrac{D}{\cdot}\right)$.
The proof of the following lemma is standard, however, for the reader's convenience, we give a proof.

\begin{lemma}\label{lem:operators}
Suppose that $f:\H\to\C$ satisfies modularity of weight $k\in\frac{1}{2}+\Z$ on $\Gamma_0(N)$ $(4 \mid N)$ with character $\chi$ of conductor $N_{\chi}$. 
\noindent

\begin{enumerate}[leftmargin=*, label=\rm(\arabic*)]
\item The function $f|U_d$ satisfies modularity of weight $k$ on $\Gamma_0(4\lcm(\frac{N}{4},\operatorname{rad}(d)))$ with character $\chi\chi_{4d}$. 
Moreover, if $f=g|V_d$, with $g$ satisfying modularity of weight $k$ on $\Gamma_0(\tfrac{N}{d})$ with character $\chi\chi_{4d}$, then $f|U_d$ satisfies modularity of weight $k$ on $\Gamma_0(\tfrac{N}{d})$ with character $\chi\chi_{4d}$. 

\item 
Suppose that $h$ satisfies modularity of weight $k$ on $\Gamma_{N,L}$ ($L\in\N$) with character $\chi$. 
For $M\not\equiv 2\pmod{4}$ (resp. $M\equiv 2\pmod{4}$), $h|S_{M,m}$ satisfies modularity of weight $k$ on $\Gamma_{\lcm(N,M^2,N_{\chi}M, ML),\lcm(M,L)}$ (resp.  $\Gamma_{\lcm(N,4M^2,N_{\chi}M,ML),\lcm(M,L)}$) with character $\chi$.  Moreover, if $M\mid 24$ and $M\not\equiv 2\pmod{4}$ (resp. $M\equiv 2\pmod{4}$), then $h|S_{M,m}$ satisfies modularity of weight $k$ on $\Gamma_{\lcm(N,M^2,MN_{\chi},ML),L}$  (resp. $\Gamma_{\lcm(N,4M^2,MN_{\chi},ML),L}$) with character $\chi$.

\item
The function $f|V_d$ satisfies modularity of weight $k$ on $\Gamma_0(Nd)$ with character $\chi\chi_{4d}$.
\item 
If $d,M\in\N$, $m\in\Z$, and $M=M_1M_2$ with $\gcd(M_1,d)=1$ and $M_2|d$, then 
\[
f\big|V_d\big|S_{M,m}= \begin{cases} f\big|S_{M_1,m\overline{d}}\big|V_d&\text{if }M_2\mid m,\\ 
0 &\text{otherwise},
\end{cases}
\]
where 
 $\overline{d}$ denotes the inverse of $d\pmod{M_1}$. If $d\mid M$, then 
\[
f\big|V_d\big|S_{M,m}= \begin{cases} f\big|S_{\frac{M}{d},\frac{m}{d}}\big|V_d&\text{if }d\mid m,\\ 
0 &\text{otherwise}.
\end{cases}
\]
For $a,d,\delta\in\N$ with $\gcd(d,\delta)=1$, we have
\[
f\big|V_{ad}\big|U_{a\delta}=f\big|U_{\delta}\big|V_{d}.
\]
For $M_1$ and $M_2$ coprime and $m_1,m_2\in\Z$, we have
\begin{equation*}
f\big|S_{M_1,m_1}\big|S_{M_2,m_2}= f\big|S_{M_2,m_2}\big|S_{M_1,m_1}, \quad 
f\big|U_{M_1}\big|S_{M_2,m_2}= f\big|S_{M_2,M_1m_2}\big|U_{M_1}.
\end{equation*}
\end{enumerate}
\end{lemma}
\begin{proof}
(1) For the first 
claim
 see for example \cite[Proposition 3.7 (2)]{OnoBook}, where for $d=\prod_{p\mid n} p^{a_p}$ we write $U_d=\prod_{p\mid n} U_p^{a_p}$ and then apply the statement iteratively for $U_p$. The second statement follows directly from the fact that $V_{d}U_d$ is the identity.

\noindent
(2) We start by rewriting
\begin{equation}\label{eqn:fsieve}
h\big|S_{M,m}(\tau) = \tfrac{1}{M}\sum_{\ell\pmod{M}}e^{\frac{2\pi i m \ell }{M}}h\left(\tau - \tfrac{\ell}{M}\right).
\end{equation}
For $\ell\in\Z$ and $\gamma=\left(\begin{smallmatrix}a&b\\ c&d\end{smallmatrix}\right)\in \Gamma_{\lcm(N,M^2,ML),L}$, we choose $\lambda\in\Z$ such that $a\lambda\equiv d\ell \pmod{M}$ and compute, using the modularity of $h$,
\begin{equation*}
h\left( \gamma\tau -\tfrac{\ell}{M}\right) =\left(\tfrac{c}{d+\lambda\tfrac{c}{M}}\right) \varepsilon_{d+\lambda\tfrac{c}{M}}^{-2k} \chi\left(d+\lambda\tfrac{c}{M}\right) 
(c\tau+d)^{k}h\left(\tau-\tfrac{\lambda}{M}\right).
\end{equation*}
Additionally assuming that $\gamma\in\Gamma_0(N_{\chi}M)$ (resp. $\gamma\in \Gamma_{0}(\lcm(4M^2,N_{\chi}M))$ if $M\not\equiv 2\pmod{4}$ (resp. $M\equiv 2\pmod{4}$), a short calculation yields
\begin{equation*}
h\left( \gamma\tau -\tfrac{\ell}{M}\right) =\left(\tfrac{c}{d}\right)\varepsilon_{d}^{-2k} \chi(d)(c\tau+d)^{k}h\left(\tau-\tfrac{\lambda}{M}\right).
\end{equation*}
Note that if $\gamma \in \Gamma_1(M)$, then we have $\lambda=\ell$. Plugging back into \eqref{eqn:fsieve} yields the first claim.

Finally, assume that $M\mid 24$. Since $\gamma\in\Gamma_0(M)$, $a$ is invertible $\pmod M$ and $a\equiv d\pmod{M}$.  Hence $\lambda=\ell$, yielding the claimed modularity in this case as well.

\noindent
(3) This is well-known; see for example \cite[Proposition 3.7 (1)]{OnoBook}.

\noindent
(4) For the first statement, we evaluate using the Chinese Remainder Theorem 
\begin{equation*}
f\big|V_d\big|S_{M,m} = f\big|V_d\big|S_{M_2,m}\big|S_{M_1,m}.
\end{equation*}
Since the $n$-th coefficient of $f\big|V_d$ vanishes unless $d\mid n$ and $M_2\mid d$, we obtain
\begin{equation*}
f\big|V_d\big|S_{M_2,m} = \begin{cases}
0\quad&\textnormal{if } M_2\nmid m,\\
f\big|V_d\quad&\textnormal{if }M_2\mid m.
\end{cases}
\end{equation*}
The first claim 
then follows by
\begin{align*}
f \big|V_d|S_{M_1,m}  = f\big|S_{M_1,m\overline d} \big|V_d,
\end{align*}
which one obtains by comparing the coefficients of the Fourier expansions on each side.

Next suppose that $d\mid M$ and write 
\begin{equation*}
f\big|V_d\big|S_{M,m}(\tau)
 = \sum_{dn\equiv m\pmod{d \frac{M}{d}}} c_{f,dv}(n)q^{dn}.
\end{equation*}
If $d\nmid m$, then the congruence is not solvable and hence $f\big|V_d\big|S_{M,m}=0$. If $d\mid m$, then the second claim follows from
\begin{equation*}
\sum_{dn\equiv d \frac{m}{d}\pmod{d \frac{M}{d}}} c_{f,dv}(n)q^{dn}= \sum_{n\equiv \frac{m}{d}\pmod{\frac{M}{d}}} c_{f,dv}(n)q^{dn} = f\big|S_{\frac{M}{d},\frac{m}{d}}\big|V_d.
\end{equation*}

We next compute the commutator relations between the $U$ and $V$ operators. Noting that $V_aU_a$ is the identity, for $\gcd(d,\delta)=1$ we have 
\begin{equation*}
f\big|V_{ad}\big|U_{a\delta}(\tau)=f\big|V_d \big|U_\delta(\tau)=\sum_{d|n} c_{f, \frac{dv}{\delta}}\left(\delta\tfrac{n}{d}\right)q^n=f\big|U_\delta V_d(\tau).
\end{equation*}
The final statement follows by the evaluation
\[
f\big|U_{M_1}|S_{M_2,m_2}(\tau)= \sum_{n\equiv m_2\pmod{M_2}} c_{f,\frac{v}{M_1}}(M_1n)q^n=f\big|S_{M_2,M_1m_2}|U_{M_1}(\tau).\qedhere
\]
\end{proof}

\subsection{Theta functions}\label{sec:genus}
Let for $\bm{a} \in \N^3$ 
\[
\ell=\ell_{\bm{a}}:=\lcm(\bm{a}) \qquad\text{ and } \qquad
\mathcal{D}_{\bm a}:=a_1a_2a_3.
\]
Using \cite[Proposition 2.1]{Shimura}, we obtain the following modularity of $\Theta_{\bm{a},\bm{h},N}$ from \eqref{3The}.
\begin{lemma}\label{lem:thetamodular}
The theta function $\Theta_{\bm{a},\bm{h},N}$ is a modular form of weight $\frac{3}{2}$ on $\Gamma_{4N^2\ell,N}$ with character $\chi_{4\mathcal{D}_{\bm a}}$. 
\end{lemma}

\subsection{Modularity of class number generating functions}
Recall that Zagier \cite{ZagierClassNum} (see also \cite[Theorem 2]{HirzebruchZagier}) has proven that the class number generating function 
\[
\mathcal{H}(\tau):=\sum_{\substack{D\geq 0\\ D\equiv 0,3\pmod{4}}} H(D) q^D
\]
satisfies modular properties. To state his result, define for $x>0$ and $\alpha \in \R$ the {\it incomplete gamma function } $\Gamma(\alpha,x):=\int_x^\infty e^{-t} t^{\alpha-1} dt$.
\begin{lemma}\label{lem:Hcomplete}
The function  
\begin{equation*}
\widehat{\mathcal{H}}(\tau):=\mathcal{H}(\tau) +\frac{1}{8\pi \sqrt{v}}+ \frac{1}{4\sqrt{\pi}}\sum_{n\geq 1}n\Gamma\left(-\tfrac12, 4\pi n^2 v\right)q^{-n^2}
\end{equation*}
transforms like a modular form of weight $\tfrac 32$ on $\Gamma_0(4)$.
\end{lemma}
The following lemma turns out to be useful for our purpose. 

\begin{lemma}\label{lem:Hslash}
	For $\ell_1,\ell_2\in\N$, with $\gcd(\ell_1,\ell_2)=1$ and $\ell_2$ squarefree,
 we have that
	\begin{equation*}
	\mathcal{H}_{\ell_1,\ell_2}:= \mathcal H\big|(U_{\ell_1\ell_2}-\ell_2U_{\ell_1}V_{\ell_2})
	\end{equation*}
	is a modular form of weight $\frac32$ on $\Gamma_0(4\operatorname{rad}(\ell_1)\ell_2)$ with character $\chi_{4\ell_1\ell_2}$.
\end{lemma}

\begin{proof}
	We first prove modularity of $\widehat{\mathcal H}_{\ell_1,\ell_2}:=\widehat{\mathcal H}\big|(U_{\ell_1\ell_2}-\ell_2U_{\ell_1}V_2)$.
	 By Lemma \ref{lem:Hcomplete}, $\widehat{\mathcal H}$ satisfies modularity of weight $\frac32$ on $\Gamma_0(4)$. Hence, by Lemma \ref{lem:operators} (1) and (3), $\widehat{\mathcal H}_{\ell_1,\ell_2}$ is modular of weight $\frac32$ on $\Gamma_0(4\operatorname{rad}(\ell_1)\ell_2)$ with character $\chi_{4\ell_1\ell_2}$.
	
	We next show that this function does not have a non-holomorphic part. 
 For this write
	\begin{equation*}
	\widehat{\mathcal H}^-(\tau)
:=\widehat{\mathcal{H}}(\tau)-\mathcal{H}(\tau) 
=: \sum_{n\geq 0} c_{v}(n)q^{-n^2} .
	\end{equation*}
	 We compute, for $\ell\in\N$,
	\begin{align*}
\widehat{\mathcal H}^-
\big|U_\ell(\tau) = \sum_{\substack{m\geq 0 \\ \ell |m^2 }} c_{\frac{v}{\ell}}(m) q^{-\frac{m^2}{\ell}}.
	\end{align*}
	Thus writing $\ell_1=\lambda_1\mu_1^2$ with $\lambda_1$ squarefree, we obtain
	\begin{equation*}
\widehat{\mathcal H}^-
\big|(U_{\ell_1\ell_2}-\ell_2U_{\ell_1}V_{\ell_2})(\tau) = \sum_{m\geq 0} c_{\frac{v}{\ell_1\ell_2}}(\lambda_1\ell_2\mu_1 m) q^{-\lambda_1\ell_2m^2} - \ell_2 \sum_{m\geq 0} c_{\frac{\ell_2v}{\ell_1}}(\lambda_1\mu_1 m) q^{-\lambda_1\ell_2m^2}.
	\end{equation*}
Plugging in the explicit formula for $c_v(n)$ from Lemma \ref{lem:Hslash}, the claim follows from the relation
	\begin{equation*}
	c_{\frac{v}{\ell_1\ell_2}}(\lambda_1\ell_2\mu_1m)-\ell_2 c_{\frac{\ell_2v}{\ell_1}}(\lambda_1\mu_1 m) 
 = 0.\qedhere
	\end{equation*}
\end{proof}

\section{Proof of Theorem \ref{thm:congcondition}}\label{sec:congcondition}

In this section we prove Theorem \ref{thm:congcondition}, reducing the task of relating $r_{\bm{a},\bm{h},N}(n)$ to class numbers to finding such formulas for $r_{\bm{a}}(n)$. 

\begin{proof}[Proof of Theorem \ref{thm:congcondition}]
We let $\bm{a}\in\mathcal{C}$ and $(\bm{h},N)\in\mathcal{S}_{\bm{a}}$. For each $m$ and $M$ appearing in the table in Appendix \ref{sec:appendixds}, for $n\equiv m\pmod{M}$ we set $d=d_{\bm{a},\bm{h},N}(n)$ to be the corresponding value of $d$ appearing in that entry from the table.

Recall that for $D\mid M$ and $j\in\Z$ and a translation-invariant function $f$, we have
\[
\sum_{\substack{m\pmod{M}\\ m\equiv j\pmod{D}}} f\big| S_{M,m} = f\big| S_{D,j}.
\]
Thus if $
d_{\bm{a},\bm{h},N}
(n)$ only depends on $n\pmod{D}$, then 
\[
\sum_{\substack{m\pmod{M}\\ m\equiv j\pmod{D}}}
d_{\bm{a},\bm{h},N}
(m)f\big| S_{M,m} =
d_{\bm{a},\bm{h},N}
(j)\sum_{\substack{m\pmod{M}\\ m\equiv j\pmod{D}}} f\big|S_{M,m} =
d_{\bm{a},\bm{h},N}
(j) f\big| S_{D,j}.
\]
This allows us to combine multiple cases by taking the lcm of the moduli of a number of cases that show up in the congruence conditions defining $
d_{\bm{a},\bm{h},N}$. 

Specifically, the claim is equivalent to showing that, for some $M_N$ with $M\mid M_N$,
\begin{equation}\label{eqn:congcondition}
\Theta_{\bm{a},\bm{h},N}=\sum_{m\pmod{M_N}} d_{\bm{a},\bm{h},N}(m) \Theta_{\bm{a}}\Big|S_{M_N,m}.
\end{equation}
We group many choices of $N$ together in each entry in Appendix \ref{sec:appendixvalence} by choosing $M_N$ as the lcm of those $M$ coming from each $(\bm{h},N)\in\mathcal{S}_{\bm{a}}$. For most cases, we prove \eqref{eqn:congcondition} directly via Lemma \ref{lem:valenceformula}. In particular if $
(\bm{a},N)\not\in
\{
((1\ 5\ 8)^T,20)
,\  
((1\ 6\ 16)^T,24)
,\ 
((1\ 9\ 21)^T,21)
,$ $ 
((1\ 9\ 24)^T,12)
,\ 
((1\ 9\ 24)^T,24)
,\ 
((1\ 16\ 24)^T,12)
,\ 
((2\ 3\ 8)^T,24)\}$, then Lemma \ref{lem:thetamodular} implies that the left-hand side of \eqref{eqn:congcondition} is a modular form of weight $\frac{3}{2}$ on  $\Gamma_{4\ell N^2,N}$ with character $\chi_{4d_{\bm{a}}}$. 
 Combining Lemmas \ref{lem:thetamodular} and \ref{lem:operators} (2), the right-hand side of \eqref{eqn:congcondition} is a modular form of weight $\frac{3}{2}$ on $\Gamma_{\lcm(4\ell, M_N^2,M_NN_{\chi_{4d_{\bm{a}}}}),M_N}$ 
with character $\chi_{4d_{\bm{a}}}$. The intersection $\Gamma$ of the two groups appearing from both sides of \eqref{eqn:congcondition} is computed for each case and listed in the row of the table in Appendix \ref{sec:appendixvalence}. Lemmas \ref{lem:valenceformula} and \ref{lem:index} yield a bound $b$ given in the last entry of the row of the table in Appendix \ref{sec:appendixvalence}, such that the identity holds if and only if holds for the first $b$ coefficients. We then check with a computer the identity for the first $b$ coefficients, yielding the claim.
For the remaining cases, $b$ could theoretically be obtained by the same method, but the computer check to verify the claim would require too much time. In these cases, we show the claim by reducing it to one of the cases which is already done.

We start with $\bm{a}=(1\ 5\ 8)^T$ and $N=20$ and reduce the claim to the claim for $N=4$ that is already proven. 
Note that 
for $\bm{h}=(0\ 4\ 5)^T$ (resp. $\bm{h}=(0\ 8\ 5)^T$) we have $n\equiv 5\pmod{25}$ (resp. $n\equiv -5\pmod{25}$)
and a direct calculation shows that 
\[
r_{{\bm{a}},\bm{h},20}(n)=\tfrac{1}{2}r_{{\bm{a}},\bm{h},4}(n).
\]
Reducing $\bm{h}\pmod{4}$ and plugging in the results for $N=4$,
 we obtain the claims.

We next consider the case $\bm{a}=(1\ 6\ 16)^T$ and $N=24$. For each of the choices in the table 
 we may show that we have
\begin{equation*}\label{eqn:1616N=24-2}
r_{{\bm{a}},\bm{h},24}(n)=\tfrac{1}{2}r_{{\bm{a}},\bm{h},8}(n).
\end{equation*}
Reducing $\bm{h}\pmod{8}$ then gives the claim.

We next assume that $\bm{a}=(1\ 9\ 21)^T$ and $N=21$. With $n$ restricted to the cases assumed in 
the table, 
 elementary congruence considerations yield
\begin{equation*}\label{eqn:toshow1921main}
r_{{\bm{a}},\bm{h},21}(n)=\tfrac{1}{2}r_{{\bm{a}},\bm{h},3}(n).
\end{equation*} 
Reducing $\bm h \pmod 3$ then gives the claim using 
the table in Appendix \ref{sec:appendixds}.

We next consider the case $\bm a=(1\ 9\ 24)^T$ and $
12\mid N
$. We
 let $\bm{h}\in\{(12\ 4\ 3)^T,(12\ 4\ 9)^T\}$ be given and 
for $N=12$ we
claim that, noting that $\bm{h}\equiv (0\ 4\ \pm 3)^T\pmod{12}$,
\begin{equation*}
r_{\bm{a},\bm{h},12}(n)=\begin{cases} \tfrac{1}{2} r_{\bm{a},\bm{h},3}(n)&\text{if }n\equiv 24\pmod{32},\\
\tfrac{1}{4}r_{\bm{a},\bm{h},3}(n)&\text{if }n\equiv 40\pmod{64},\\
0&\text{otherwise}.
\end{cases}
\end{equation*}
In terms of theta functions, this claim may be written as
\begin{equation}\label{eqn:19243to12}
\Theta_{\bm{a},\bm{h},12}= \tfrac{1}{2}\Theta_{\bm{a},\bm{h},3}\big|S_{32,24}+\tfrac{1}{4}\Theta_{\bm{a},\bm{h},3}\big|S_{64,40}.
\end{equation}
To prove \eqref{eqn:19243to12}, note that by Lemma \ref{lem:thetamodular}, the left-hand side of \eqref{eqn:19243to12} is a modular form of  weight $\frac{3}{2}$ on $\Gamma_{41472,12}$ with character $\chi_{24}$
 For the right-hand side, we use that $\Theta_{\bm{a},\bm{h},3}$ is a modular form of weight $\frac{3}{2}$ on $\Gamma_{2592,3}$ with character $\chi_{24}$. Applying Lemma \ref{lem:operators} (2) yields a modular form of weight $\frac{3}{2}$ on $\Gamma_{331776,192}$ 
with character $\chi_{24}$.
 By Lemmas \ref{lem:valenceformula} and \ref{lem:index} we need to check the first $5\,308\,416$ coefficients.
To complete the $N=12$ case, 
we note that $\bm{h}\equiv (0\ 1\ 0)^T\pmod{3}$, so we may plug in the identity from the $N=3$ cases from the table.

 We next claim that for $N=24$ we have 
\begin{equation*}
r_{\bm{a},\bm{h},24}(n)=\begin{cases} r_{\bm{a},\bm{h},12}(n)&\text{if }n\equiv 120\pmod{128}\text{ and }\bm{h}=(12\ 4\ 3)^T\\
&\text{ or }n\equiv 56\pmod{128}\text{ and }\bm{h}=(12\ 4\ 9)^T,\\
0&\text{otherwise}.
\end{cases}
\end{equation*}
Note that the claim is equivalent to
\begin{equation}
\label{eqn:192412to24} \Theta_{\bm{a},\bm{h},24}= 
\begin{cases}
\Theta_{\bm{a},\bm{h},12}\big|S_{128,120}&\text{if }\bm{h}=(12\ 4\ 3)^T,\\
\Theta_{\bm{a},\bm{h},12}\big|S_{128,56}&\text{if }\bm{h}=(12\ 4\ 9)^T.\\
\end{cases}
\end{equation}
To prove \eqref{eqn:192412to24}, use that the left-hand side is by 
Lemma \ref{lem:thetamodular} a
 modular form of weight $\frac32$ on $\Gamma_{165888,24}$ with character $\chi_{24}$ which is trivial on that group. For the right-hand side $\Theta_{\bold{a}, \bold{h}, 12}$ is a modular form of weight $\frac32$ on $\Gamma_{41472,12}$ 
with character $\chi_{24}$.
By Lemma \ref{lem:operators} (2), $\Theta_{\bm{a},\bm{h},12}\big|S_{128,m}$ is a modular form of weight $\frac{3}{2}$ on $ \Gamma_{1327104,384}$. By Lemmas \ref{lem:valenceformula} and \ref{lem:index}, we need to check the first $42\,467\,328$ coefficients. 
Plugging in the identities for $N=12$ yields the claim.

For $\bm{a}=(1\ 16\ 24)^T$ and $N=12$,
we can show that
\[
r_{{\bm{a}},\bm{h},12}(n)=\tfrac{1}{2}r_{{\bm{a}},\bm{h},4}(n).
\]
Reducing $\bm{h}\pmod{4}$ and plugging in the formula for $N=4$ yields the claim.

For $\bm{a}=(2\ 3\ 8)^T$, $N=24$, and $n\equiv 3\pmod{8}$ we have 
\[
r_{\bm{a},\bm{h},24}(n)=\tfrac{1}{2}r_{\bm{a},\bm{h},8}(n).
\]
The table in Appendix \ref{sec:appendixds} for $\bm{a}=(2\ 3\ 8)^T$
 gives the claim reducing $\bm h \pmod 8$.
\end{proof}

\section{Class number identities}\label{sec:lemmas}
In this section, we derive relations like \eqref{eqn:r1notGauss} between $r_{\bm{a}}(n)$ and Hurwitz class numbers for each form $\bm{a}\in\mathcal{C}$.  The general idea is as follows: We use Lemmas \ref{lem:thetamodular}, \ref{lem:Hslash}, and \ref{lem:operators} to show that both sides of the identity are modular forms for some congruence subgroup $\Gamma$ of $\Gamma_0(4)$ and then use Lemmas \ref{lem:valenceformula} and \ref{lem:index} to prove that the claim follows as long as it holds for the first $b$ (given in Appendix \ref{sec:appendixvalence}) Fourier coefficients. Finally, we check $b$ coefficients with a computer.  We give a detailed proof of one case in Lemma \ref{lem:111} in order to show how to carry out the details to obtain $\Gamma$ and $b$ and leave the details for the other cases to the reader.

Setting $c_{(1\ 1\ 1)^T}(n) :=12$ and taking $c_{\bm{a}}(n)$ from the tables in Appendix \ref{sec:appendixcs} otherwise, we show  the following.
\begin{lemma}\label{lem:111}
For $\bm{a}\in \{(1\ 1\ 1)^T,(1\ 1\ 4)^T, (1\ 2\ 2)^T, (1\ 2\ 8)^T, (1\ 4\ 4)^T, (1\ 8\ 8)^T\}$, we have 
\[
r_{{\bm{a}}}(n)= c_{\bm a}(n)\left(H(4n)-2H(n)\right).
\]
\end{lemma}
\begin{proof}[Proof of Lemma \ref{lem:111}]
For each choice of $\bm{a}$, we let $M$ be the lcm of the corresponding entries in the tables in Appendix \ref{sec:appendixcs}. It is then not hard to see that the claim is equivalent to
\[
\Theta_{{\bm{a}}} = \sum_{m \pmod{M}} c_{\bm a} (m) \mathcal{H}_{1,2} \big| U_2 \big|S_{M,m}.
\]
To prove this, note that Lemma \ref{lem:thetamodular} implies that the left-hand side is a modular form of weight $\frac{3}{2}$ on $\Gamma_0({4 \ell})$ ($\ell\mid 8$). By Lemma \ref{lem:thetamodular}, $\mathcal H_{1,2}$ is a modular form of weight $\frac 32$ on $\Gamma_0(8)$ with character $\chi_8$. Hitting with $U_2$ gives by Lemma \ref{lem:operators} (1) a modular form of weight $\frac 32$ on $\Gamma_0(8)$. For simplicity we only consider $\bm a =(1\ 2\ 8)^T$. Hitting with $S_{16,m}$ gives a modular form of weight $\frac3 2$ on $\Gamma_{256,16}$.
 By Lemmas \ref{lem:valenceformula} and \ref{lem:index}, it suffices to check the identity for the first 
\[
\tfrac{1}{8}\left[\SL_2(\Z):\Gamma_{256,16}\right]=384
\]
coefficients, which is easily verified.
\end{proof}

\begin{lemma}\label{lem:112}
For $\bm{a}\in \{(1\ 1\ 2)^T, (1\ 1\ 8)^T, (1\ 2\ 4)^T, (1\ 2\ 16)^T, (1\ 4\ 8)^T,  (1\ 8\ 16)^T\}$, we have 
\[
r_{{\bm{a}}}(n)=c_{\bm a}(n)\left( H(8n) -2 H(2n) \right).
\]
\end{lemma}
\begin{proof} It is not hard to see that the formula is equivalent to 
	\[
	\Theta_{{\bm{a}}}=\sum_{m\pmod{M}} 
	c_{\bm a}(m)
	\mathcal{H}_{1,2}\big|U_4\big|S_{M,m}.
	\]	
	The table in Appendix \ref{sec:appendixvalencelemmas}, together with a short computer check, finishes, in this case, as well as in the following ones, the claim.
\end{proof}

Letting $c_{(1\ 1\ 3)^T}(n) :=2$, we have 
the following.
\begin{lemma}\label{lem:113}
	For $\bm a \in\{ (1\ 1\ 3)^T, (1\ 3\ 9)^T\}$, we have
	\[
	r_{{\bm a}}(n)= c_{\bm a}(n) \left( H(12n)+2 H(3n)-3H \left(\tfrac{4n}{3}\right)- 6 H\left(\tfrac{n}{3}\right)\right).
	\]
\end{lemma}
\begin{proof}
The equivalent formula in this case is 
		\[
	\Theta_{{\bm{a}}}=\sum_{m\pmod M} c_{\bm a}(m)
\mathcal{H}_{1,3}
\big|\left(2+U_4\right)\big|S_{M,m}. \qedhere
	\]
\end{proof}

\begin{lemma}\label{lem:115}
For $\bm{a}=(1\ 1\ 5)^T$, we have,
\[
r_{{\bm{a}}}(n)=2 
H(20n) -4 H(5n)
 +10 H \left( \tfrac{4n}{5} \right) -20 H \left( \tfrac n5 \right)  .
\]
\end{lemma}
\begin{proof}
	The equivalent formula in this case is
	\[
	\Theta_{{\bm{a}}}=2\left(\mathcal{H}_{5,2}+5\mathcal{H}_{1,2}\big|V_5\right)\big|U_2. \qedhere
	\]
\end{proof}

Let $c_{(1\ 1\ 6)^T}(n):=2(-1)^{n+1}$.
\begin{lemma}\label{lem:116}
For $\bm{a}\in \{(1\ 1\ 6)^T,(1\ 6\ 9)^T\}$, we
 have
\[
r_{{\bm{a}}}(n)=c_{\bm{a}}(n) \left(H(24n) -4 H(6n) -3 H \left( \tfrac{8n}{3} \right) +12 H \left( \tfrac{2n}{3} \right) \right).
\]
\end{lemma}
\begin{proof}
	The equivalent formula in this case is
		\[
	\Theta_{{\bm{a}}}= -\sum_{m\pmod{M}} c_{\bm a}(m)\mathcal{H}_{2,3}\big|\left(4-U_4\right)\big|S_{M,m}.\qedhere
	\]
\end{proof}

\begin{lemma}\label{lem:119}
We have, for $\bm{a}=(1\ 1\ 9)^T$,
\[
r_{{\bm{a}}}(n)= 12  H \left( \tfrac{4n}{9} \right) -24H\left(\tfrac{n}{9}\right)  +c(n) (H(4n) -2H(n)).
\]

\end{lemma}
\begin{proof}
	The equivalent formula in this case is
	\[
	\Theta_{{\bm{a}}}=12\mathcal{H}_{1,2}\big|U_2\big|V_9 + \sum_{m\pmod{3}}c(m) \mathcal{H}_{1,2}\big|U_2\big|S_{3,m}.\qedhere
	\]
\end{proof}

\begin{lemma}\label{lem:1112}
For $\bm{a}\in \{(1\ 1\ 12)^T,(1\ 4\ 12)^T,(1\ 9\ 
12
)^T\}$, we have 
\[
r_{{\bm{a}}}(n)=c_{\bm a}(n)\left(H(12n) + 8 H \left( \tfrac{3n}{4} \right) -3 H \left( \tfrac{4n}{3} \right) -24 H \left( \tfrac{n}{12} \right)\right).
\]
\end{lemma}
\begin{proof}
	The equivalent formula in this case is 
	\[
	\Theta_{{\bm{a}}} = \sum_{m\pmod{M}} c_{\bm a}(m)\mathcal{H}_{1,3}\big|\left(U_4+8V_4\right)\big|S_{M,m}.\qedhere
	\]
\end{proof}

\begin{lemma}\label{lem:1121}
For $\bm{a}=(1\ 1\ 21)^T$, we have 
\begin{multline*}
r_{{\bm{a}}}(n)=c(n)
\left(
H(84n) -7 H \left( \tfrac{12n}{7} \right) -3 H \left( \tfrac{28n}{3} \right) +21 H \left( \tfrac{4n}{21} \right) -4 H \left( \tfrac{21n}{4} \right) +28 H \left( \tfrac{3n}{
	28
} \right) 
\right. \\
\left.
+12
H\left(\tfrac{7n}{
	12
}\right) -84 H \left( \tfrac{n}{84} \right)
\right).
\end{multline*}

\end{lemma}
\begin{proof}
The equivalent formula in this case is
	\[
	\Theta_{{\bm{a}}}=\sum_{m\pmod{8}} c(m)\left(\mathcal{H}_{12,7}-3\mathcal{H}_{4,7}\big|V_3-4\mathcal{H}_{3,7}\big|V_4+12\mathcal{H}_{1,7}\big|V_{12}\right)\big|S_{8,m}.\qedhere
	\]
\end{proof}

\begin{lemma}\label{lem:1124}
For $\bm{a}\in\{(1\ 1\ 24)^T, (1\ 9\ 24)^T\}$, we have 
\[
r_{{\bm{a}}}(n)=c_{\bm a}(n)\left( H(24n)-H(6n)-3H\left(\tfrac{8n}{3}\right)+3H\left(\tfrac{2n}{3}\right) +12H\left(\tfrac{3n}{8}\right)- 36H\left(\tfrac{n}{24}\right)\right).
\]
\end{lemma}
\begin{proof}
The equivalent formula in this case is
	\[
	\Theta_{{\bm{a}}} = 
	\sum_{m\pmod{
M
}} c_{\bm a}(m)\left(\mathcal{H}_{8,3} -\mathcal{H}_{2,3} +12\mathcal{H}_{1,3}\big| V_8\right)\big|S_{
M,m
}.\qedhere
	\]
\end{proof}

\begin{lemma}\label{lem:123}
For $\bm{a}=(1\ 2\ 3)^T$, we have 
\[
r_{{\bm{a}}}(n)=c(n)\left( H(24n)-2 H \left(6n\right) +3 H \left( \tfrac{8n}{3} \right)  -6H \left( \tfrac{2n}{3} \right) \right).
\]

\end{lemma}
\begin{proof}
The equivalent formula in this case is
	\[
	\Theta_{{\bm{a}}}= \sum_{m\pmod{2}}c(m)\mathcal{H}_{1,2}\big|\left(U_3+3V_3\right)\big|U_4\big|S_{2,m}.\qedhere
	\]
\end{proof}

\begin{lemma}\label{lem:125}
	
For $\bm{a}=(1\ 2\ 5)^T$, we have
\[
r_{{\bm{a}}}(n)= (-1)^{n+1}\left(H(40n)-4H(10n)-5H\left(\tfrac{8n}{5}\right)+20H\left(\tfrac{2n}{5}\right)\right).
\]

\end{lemma}
\begin{proof}
The equivalent formula in this case is
	\[
	\Theta_{\bm{a}}=\sum_{m\pmod{2}} (-1)^{m+1}\left(\mathcal{H}_{8,5}-4\mathcal{H}_{2,5}\right)\big|S_{2,m}.\qedhere
	\]
\end{proof}

\begin{lemma}\label{lem:126}
We have, for $\bm a = (1\ 2\ 6)^T$,
\begin{equation*}
r_{{\bm{a}}}(n) =  c(n) \left( H(12n) -2 H(3n) +3 H \left( \tfrac{4n}{3} \right) - 6 H \left( \tfrac{n}{3} \right)\right).
\end{equation*}
\end{lemma}

\begin{proof}
The equivalent formula in this case is
	\[
	\Theta_{{\bm{a}}}=\sum_{m\pmod{4}} c(m)\left( \mathcal{H}_{3,2} +3\mathcal{H}_{1,2}\big|V_3\right)\big|U_{2}\big|S_{4,m}.\qedhere
	\]
\end{proof}

\begin{lemma}\label{lem:1210}
For $\bm{a}=(1\ 2\ 10)^T$, we have 
\[
r_{{\bm{a}}}(n)=c(n) \left( H(20n) -4 H(5n)-5 H \left( \tfrac{4n}{5} \right) +20 H \left( \tfrac{n}{5} \right)
\right).
\]
\end{lemma}
\begin{proof}
The equivalent formula in this case is
	\[
	\Theta_{{\bm{a}}} = \sum_{m\pmod{4}} c(m)\left( \mathcal{H}_{4,5}-4\mathcal{H}_{1,5}\right)\big|S_{4,m}.\qedhere
	\]
\end{proof}

\begin{lemma}\label{lem:133}
For $\bm{a}=(1\ 3\ 3)^T$, we have 
\[
r_{{\bm{a}}}(n)=2 H(36n) +4 H(9n) -6 H(4n) - 12 H(n) .
\]
\end{lemma}
\begin{proof}
The equivalent formula in this case is
	\[
	\Theta_{{\bm{a}}}=2\left(\mathcal{H}_{4,3} + 2\mathcal{H}_{1,3}\right)\big|U_3.\qedhere
	\]
\end{proof}

\begin{lemma}\label{lem:134}
For $\bm{a}=(1\ 3\ 4)^T$, we have 
\[
r_{{\bm{a}}}(n)=
c(n)
 \left( H(12n) -3 H \left( \tfrac{4n}{3} \right) + 2 H \left( \tfrac{3n}{4} \right) - 6 H \left( \tfrac{n}{12} \right) \right).
\]
\end{lemma}

\begin{proof}
The equivalent formula in this case is
	\[
	\Theta_{{\bm{a}}}=\sum_{m\pmod{8}}
	c(m)
	\left(\mathcal{H}_{4,3} + 2\mathcal{H}_{1,3}\big|V_4\right)\big|S_{8,m}.\qedhere
	\]
\end{proof}

\begin{lemma}\label{lem:136}
For $\bm{a}=(1\ 3\ 6)^T$, we have 
\[
r_{\bm{a}}(n)=c(n) \left(
-H(72n)+2H(18n)+5H(8n)-10H(2n)
\right).
\]
\end{lemma}
\begin{proof}
The equivalent formula in this case is
	\[
	\Theta_{{\bm{a}}}=\sum_{m\pmod{2}} c(m) \left(
5\mathcal{H}_{1,2}-\mathcal{H}_{9,2}
\right)\big| U_4\big| S_{2,m}.\qedhere
	\]
\end{proof}

\begin{lemma}\label{lem:1310}
	For $\bm{a}=(1\ 3\ 10)^T$, we have 
	\begin{align*}
	r_{{\bm a}}(n) = \tfrac 12 H(120n)-2H\left(\tfrac{15n}{2}\right)-\tfrac 52 H\left(\tfrac{24n}{5}\right)+10 H\left(\tfrac{3n}{10}\right)-\tfrac 32 H\left(\tfrac{40 n}{3}\right)+6H\left(\tfrac{5n}{6}\right)\\
	+\tfrac{15}{2} H\left(\tfrac{8n}{15}\right)-30 H\left(\tfrac{n}{30}\right).
	\end{align*}
\end{lemma}
\begin{proof}
The equivalent formula in this case is
	\[
	\Theta_{{\bm{a}}}=\tfrac{1}{2}\mathcal{H}_{1,3}\big| \left(U_5-5V_5\right)\big|\left(U_8-4V_2\right).\qedhere
	\]
\end{proof}

\begin{lemma}\label{lem:1312}
	We have, for $\bm{a}= (1\ 3\ 12)^T$,
	\begin{equation*}
	r_{{\bm a}}(n)= c(n) \left(H(36n)-3H(4n)+2
	H\left(\tfrac{9n}{4}\right)
	-6H\left(\tfrac{n}{
		4
	}\right)\right).
	\end{equation*}
\end{lemma}

\begin{proof}
The equivalent formula in this case is
	\[
	\Theta_{{\bm{a}}}=\sum_{m\pmod{8}} 
	c(m)
	\mathcal{H}_{1,3}\big| \left(U_{12}+2U_{3}V_4\right) \big|S_{8,m}.\qedhere
	\]
\end{proof}

\begin{lemma}\label{lem:1318}
	For $\bm{a}=(1\ 3\ 18)^T$, we have 
\[
r_{{\bm a}}(n)= c(n) \left(H(24n)-2H(6n) -5H\left(\tfrac{8n}{3}\right) +10 H \left(\tfrac{2n}{3} \right) \right).
\]
\end{lemma}

\begin{proof}
The equivalent formula in this case is
	\[
	\Theta_{{\bm{a}}}= \sum_{m\pmod{6}} c(m)\mathcal{H}_{1,2}\big|\left(U_{3}- 5V_3\right)\big|U_4 \big|S_{6,m}. \qedhere
	\]
\end{proof}

\begin{lemma}\label{lem:1330}
	We have, for $\bm a= (1\ 3\ 30)^T$,
	\begin{multline*}
	r_{{\bm a}}(n)=c(n)\Big(H(360n)-2H(90n)-3H(40n)+6H(10n)-5H\left(\tfrac{72n}{5}\right)+15H\left(\tfrac{8n}{5}\right)\\	
+10H\left(\tfrac{18n}{5}\right)-30H\left(\tfrac{2n}{5}\right)\Big).
	\end{multline*}
\end{lemma}

\begin{proof}
The equivalent formula in this case is
	\[
	\Theta_{{\bm{a}}}=\sum_{m\pmod{2}} 
	c(m)
	\left(\mathcal{H}_{40,3} -5
	\mathcal{H}_{8,3}
	\big|V_5-2\mathcal{H}_{20,3}|V_2+10\mathcal{H}_{4,3}|V_{10}\right)\big|U_3\big| S_{2,m}. 
	\]
	The claimed identity follows using that $\mathcal{H}(n)=0$ if $2\parallel n$.\qedhere
\end{proof}

\begin{lemma}\label{lem:146}
For $\bm{a}\in\{(1\ 4\ 6)^T, (1\ 4\ 24)^T\}$,  we have 
\[
r_{{\bm{a}}}(n)= c_{\bm a}(n) \left(2 H(24n) -5 H(6n) -6 H \left( \tfrac{8n}{3} \right) +15H\left( \tfrac{2n}{3} \right)\right).
\]
\end{lemma}
\begin{proof}
The equivalent formula in this case is
	\[
	\Theta_{{\bm{a}}}=\sum_{m\pmod{8}} c_{\bm a}(m)\mathcal{H}_{2,3}\big|\left(
-5+2U_4
\right)\big|S_{8,m}. \qedhere
	\]
\end{proof}

\begin{lemma}\label{lem:155}
	For $\bm{a}=(1\ 5\ 5)^T$, we have
\[
r_{{\bm{a}}}(n)=-2 H(100n) +4 H (25n) +14 H(4n) -28 H(n)
.
\]
\end{lemma}
\begin{proof}
The equivalent formula in this case is
	\[
	\Theta_{{\bm{a}}} = -2 \left(\mathcal{H}_{25,2}-7\mathcal{H}_{1,2}\right)\big| U_2. \qedhere
	\]
\end{proof}

\begin{lemma}\label{lem:158}
 We have, for $\bm a =(1\ 5\ 8)^T$,
\[
 r_{{\bm a}}(n)= c(n)\left(  H(160n)- H(40n)-5H\left(\tfrac{32n }{5}\right)  +  5H \left( \tfrac{8n}{5} \right)-12 \delta_{8\mid n} \left( H\left(\tfrac{5n}{2}\right) -5 H\left(\tfrac{n}{10}\right)\right)\right).
\]
\end{lemma}
\begin{proof}
The equivalent formula in this case is
	\[
	\Theta_{{\bm{a}}}=\sum_{m\pmod{8}} c(m)\left( \mathcal{H}_{32,5}-\mathcal{H}_{8,5}-12\mathcal{H}_{2,5}\big|V_4\right)\big|S_{8,m}. \qedhere
	\]
\end{proof}

\begin{lemma}\label{lem:1510}
	We have, for $\bm a = (1\ 5\ 10)^T$,
	\begin{equation*}
	r_{{\bm a}}(n)= 
(-1)^{n+1}
\left( H(200n)-4 H(50n)- 5 H(8n) 
 +20 H(2n) \right).
	\end{equation*}
\end{lemma}

\begin{proof}
The equivalent formula in this case is
	\[
	\Theta_{{\bm{a}}}=\sum_{m\pmod{2}} 
	(-1)^{m}
	\mathcal{H}_{2,5}\big|\left(4-U_{4}\right)\big|U_5\big|S_{2,m}. \qedhere
	\]
\end{proof}

\begin{lemma}\label{lem:1525}
	We have for, $\bm a = (1\ 5\ 25)^T$,
	\begin{equation*}
	r_{{\bm a}}(n)= c(n)\left( H(20n)- 2 H(5n)  - 7H\left(\tfrac{4n}{5}\right) +14 H\left(\tfrac{n}{5}\right)\right).
	\end{equation*}
\end{lemma}
\begin{proof}
The equivalent formula in this case is
	\[
	\Theta_{{\bm{a}}}=\sum_{m\pmod{5}} c(m) \left(\mathcal{H}_{5,2}-7\mathcal{H}_{1,2}\big|V_5\right)\big|U_2\big|S_{5,m}. \qedhere
	\]
\end{proof}

\begin{lemma}\label{lem:1540}
	We have, for $\bm a = (1\ 5\ 40)^T$,
		\begin{equation*}
	r_{{\bm a}}(n)= c(n)\left(
{
H(200n)- H(50n)-5
}
H(8n) 
{
+5
}
H(2n)
{
-6
}
\delta_{8\mid n}
\left(H\left(\tfrac{25n}{2}\right)-5 H\left(\tfrac{n}{2}\right)\right)\right).
	\end{equation*}
\end{lemma}

\begin{proof}
The equivalent formula in this case is
	\[
	\Theta_{{\bm{a}}}=-\sum_{m\pmod{8}} c(m)\mathcal{H}_{2,5}\big|\left(
{
1-U_{4}+6V_4
}
\right)\big|U_5\big|S_{8,m}. \qedhere
	\]
\end{proof}

\begin{lemma}\label{lem:166}
	We have, for $\bm a \in \{ (1\ 6\ 6)^T,(1\ 12\ 12)^T\}$,
	\begin{equation*}
	r_{{\bm a}}(n)= 
c_{\bm{a}}(n)
\left( H(36n)-3H(4n) - 4 H(9n)+12 H(n)\right).
	\end{equation*}
\end{lemma}

\begin{proof}
The equivalent formula in this case is 
	\[
	\Theta_{{\bm{a}}}=\sum_{m\pmod{4}} 
c_{\bm{a}}(n)
\left(\mathcal{H}_{4,3}-4\mathcal{H}_{1,3}\right)\big|U_3\big|S_{4,m}. \qedhere
	\]
\end{proof}

\begin{lemma}\label{lem:1616}
	We have, for $\bm a \in\{ (1\ 6\ 16)^T, (1\ 16\ 24)^T\}$,	
\[
r_{{\bm{a}}}(n)=c_{\bm a}(n)\left( H(24n)-H(6n)-3H\left(\tfrac{8n}{3}\right)+3H\left(\tfrac{2n}{3}\right) -12H\left(\tfrac{3n}{8}\right)+ 36H\left(\tfrac{n}{24}\right)\right).
\]
\end{lemma}
\begin{proof}
The equivalent formula in this case is
\[
\Theta_{\bm{a}}=\sum_{m\pmod{32}}c_{\bm a} (m) \mathcal{H}_{1,3}\big|\left(U_8-U_2-12V_8\right)\big|S_{32,m}.\qedhere
\]
\end{proof}

\begin{lemma}\label{lem:1618}
For $\bm a =(1\ 6\ 18)^T$, we have 
\[
r_{{\bm a}}(n)=  c(n) \left( H(12n) -2 H(3n) -5 H \left(\tfrac{4n}{3} \right) + 10 H \left( \tfrac{n}{3} \right)\right).
\]
\end{lemma}
\begin{proof}
The equivalent formula in this case is 
	\[
	\Theta_{{\bm{a}}}=\sum_{m\pmod{12}} c(m)\left(\mathcal{H}_{3,2}-5\mathcal{H}_{1,2}|V_{3}\right)\big|U_2 \big|S_{12,m}. \qedhere
	\]
\end{proof}

\begin{lemma}\label{lem:1624}
For $\bm a=(1\ 6\ 24)^T$, we have
\[
r_{{\bm a}}(n) = c(n)\left( 2H(36n)-6 H(4n)-5H\left(9n\right) +15H\left(n\right)\right).
\]
\end{lemma}
\begin{proof}
The equivalent formula in this case is
	\[
	\Theta_{{\bm{a}}}=\sum_{m\pmod{16}}c(m)\left(2\mathcal{H}_{4,3}-5\mathcal{H}_{1,3}\right)\big|U_3\big|S_{16,m}. \qedhere
	\]
\end{proof}

\begin{lemma}\label{lem:1840}
	We have, for $\bm a =(1\ 8\ 40)^T$,
	\begin{equation*}
	r_{{\bm a}}(n)= c(n) \left( 2 H(20n) -5H(5n)-10 H \left( \tfrac{4n}{5} \right) +25 H \left( \tfrac n5 \right)\right).
	\end{equation*}	
\end{lemma}
\begin{proof}
The equivalent formula in this case is
	\[
	\Theta_{{\bm{a}}}=\sum_{m\pmod{16}} c(m)\left(2\mathcal{H}_{4,5}-5\mathcal{H}_{1,5}\right)\big|S_{16,m}. \qedhere
	\]
\end{proof}

\begin{lemma}\label{lem:199}
	We have, for $\bm a =(1\ 9\ 9)^T$,
	\begin{equation*}
	r_{{\bm a}}(n)=  12 H \left( \tfrac{4n}{9} \right) -24 H \left( \tfrac n9 \right) + 4\delta_{n \equiv 1 \pmod 3} \left(H(4n) -2 H(n) 
	\right).
	\end{equation*} 
\end{lemma}
\begin{proof}
The equivalent formula in this case is
	\[
	\Theta_{{\bm{a}}}=
	4 \mathcal H_{1,2} \big| U_2\big| (3 V_9+S_{3,1}). \qedhere
	\]
\end{proof}

{

\begin{lemma}\label{lem:1921}
	We have, for $\bm{a}=(1\ 9\ 21)^T$,
	\begin{multline*}
	r_{{\bm a}}(n)=c(n) \Big( H(84n)-  2H(21n) -7 H \left( \tfrac{12n}{7} \right)  +14 H \left( \tfrac{3n}{7} \right) -3 H \left( \tfrac{28n}{3} \right) \\ + 21 H \left( \tfrac{4n}{21} \right)+6 H\left(\tfrac{7n}{3}\right)- 42 H \left( \tfrac{n}{21} \right)
	\Big).
	\end{multline*}
\end{lemma}
\begin{proof}
The equivalent formula in this case is 
		\[
	\Theta_{{\bm{a}}}=-\sum_{m\pmod{
3
}} c(m)\mathcal{H}_{1,3}\big|\left(2-U_4\right)\big|\left(U_7-7V_7\right)\big|S_{
3,m
}. \qedhere
	\]
\end{proof}

\begin{lemma}\label{lem:12121}
	
	We have, for  $\bm{a}=(1\ 21\ 21)^T$,
\begin{multline*}
r_{{\bm a}}(n)=H(1764n)- 2H(441n) -3H(196n) + 6H(49n) - 7H(36n) +14H(9n) \\+21H(4n) -42H(n).
\end{multline*}
\end{lemma}
\begin{proof}
The equivalent formula in this case is
	\[
	\Theta_{\bm{a}}=\mathcal{H}_{1,7}\big|\left(2-U_4\right)\big|\left(3-U_9\right)\big|U_7.\qedhere
	\]
\end{proof}

\begin{lemma}\label{lem:12424}  
	We have, for  $\bm{a}=(1\ 24\ 24)^T$,
\[
r_{{\bm a}}(n)=c(n) \left( 2H(36n) -5H(9n)-6 H(4n) +15H(n)\right).
\]
\end{lemma}
\begin{proof}
The equivalent formula in this case is
	\[
	\Theta_{\bm{a}}=-\sum_{m\pmod{16}}
	c(m)
	\mathcal{H}_{1,3}\big|\left(5-2U_4\right)\big|U_3\big|S_{16,m}.\qedhere
	\]
\end{proof}

\begin{lemma}\label{lem:223}
	For $\bm{a}\in\{(2\ 2\ 3)^T, (2\ 3\ 18)^T, (3\ 4\ 4)^T, (3\ 4\ 36)^T
\}$, we have 
\[
r_{{\bm{a}}}(n)=c_{\bm a}(n)\left( H(12n) -4 H(3n) -3 H \left( \tfrac{4n}{3} \right)+12H\left(\tfrac{n}{3}\right)\right).
\]
\end{lemma}
\begin{proof}
The equivalent formula in this case is
\[
\Theta_{{\bm{a}}}=-\sum_{m\pmod{M}} c_{\bm a}(m)\mathcal{H}_{1,3}\big|\left(4-U_4\right)\big|S_{M,m}. \qedhere
\]
\end{proof}

 \begin{lemma}\label{lem:233}
 	For $\bm{a}=(2\ 3\ 3)^T$, we have
 	\[
 	r_{{\bm{a}}}(n)=2 (-1)^n \left( H(72n) - 4  H(18n) -3 H(8n) + 12 H(2n)\right).
 	\]
 \end{lemma}
 \begin{proof}
The equivalent formula in this case is
 		\begin{equation*}
 	\Theta_{{\bm{a}}}=  2\sum_{m\pmod{2}}(-1)^m \mathcal{H}_{4,3}\big|\left(U_2 -4V_2\right)\big|U_3 \big|S_{2,m}. \qedhere
 	\end{equation*}
 \end{proof}

\begin{lemma}\label{lem:236}
For $\bm{a}=(2\ 3\ 6)^T$, we have
\[
r_{{\bm a}}(n)=c(n)\left(H(36n)-2H(9n)-5H(4n)+10H(n)\right).
\]
\end{lemma}
\begin{proof}
The equivalent formula in this case is
\[
\Theta_{{\bm{a}}}=-\sum_{m\pmod{4}}c(m)\mathcal{H}_{1,2}\big|\left(5-U_9\right)\big|U_2\big|S_{4,m}. \qedhere
\]
\end{proof}

\begin{lemma}\label{lem:238}
	We have, for $\bm a \in \{ (2\ 3\ 8)^T, (3\ 8\ 8)^T, (3\ 8\ 72)^T\}$,
	\[
	r_{{\bm a}}(n)=c_{\bm a}(n)\left( 2H(12n) -5 H(3n)-6 H\left( \tfrac{4n}{3}\right)  +15 H \left( \tfrac{n}{3}\right) \right).
	\]
\end{lemma}
\begin{proof}
The equivalent formula in this case is
	\[
	\Theta_{{\bm{a}}}=-\sum_{m\pmod{M}} c_{\bm a}(m)\mathcal{H}_{1,3}\big|\left(5-2U_4\right)\big|S_{
M,m
}. \qedhere
	\]
\end{proof}

\begin{lemma}\label{lem:239}
	For $\bm{a}=(2\ 3\ 9)^T$, we
	have
	\[
	r_{{\bm{a}}}(n)= c(n)\left(H(24n)-2H(6n)-5H\left(\tfrac{8n}{3}\right)+10H\left(\tfrac{2n}{3}\right)\right).
	\]
\end{lemma}
\begin{proof}
The equivalent formula in this case is
\[
\Theta_{{\bm{a}}}=\sum_{m\pmod{6}} c(m)\mathcal{H}_{1,2}\big|\left(U_3-5V_3\right)\big|U_4\big|S_{6,m}.\qedhere
\]
\end{proof}

\begin{lemma}\label{lem:2312}
	For $\bm{a}\in\{(2\ 3\ 12)^T, (3\ 8\ 12)^T\}
 $, we have 
\begin{equation*}
r_{{\bm a}} (n)=
 c_{\bm{a}}(n)
\left(2H(72n) -5H(18n)-6 H(8n) +15H(2n)\right).
\end{equation*}
\end{lemma}
\begin{proof}
The equivalent formula in this case is
	\[
	\Theta_{{\bm{a}}}= \sum_{m\pmod{8}}
c_{\bm{a}}(m)
\left(2\mathcal{H}_{8,3}-5\mathcal{H}_{2,3}\right)\big|U_3 \big|S_{8,m}. \qedhere
	\]
\end{proof}

\begin{lemma}\label{lem:2348}
For $\bm{a}=(2\ 3\ 48)^T$, we have
\begin{equation*}
r_{{\bm a}}(n)=
c(n)
\left(H\left(72n\right)-H(18n)-3H\left(8n\right)+3H(2n)-12H\left(\tfrac{9n}{8}\right)+36H\left(\tfrac{n}{8}\right)\right).
\end{equation*}
\end{lemma}
\begin{proof} 
The equivalent formula in this case is
	\[
	\Theta_{{\bm{a}}}=\sum_{m\pmod{32}} 
c(m)
\left(\mathcal{H}_{8,3}-\mathcal{H}_{4,3}\big|V_2-12\mathcal{H}_{1,3}\big|V_8\right)\big|U_3\big|S_{32,m}. \qedhere
	\]
\end{proof}

\begin{lemma}\label{lem:256}
	We have, for $\bm a = (2\ 5\ 6)^T$,
	\begin{multline*}
	r_{{\bm a}}(n)=c(n) \left( H(60n)  -2 H(15n) - 5 H\left( \tfrac{12n}{5} \right) -3 H \left( \tfrac{20n}{3} \right)  +10 H \left( \tfrac{3n}{5} \right) +6 H \left( \tfrac{5n}{3} \right) \right.\\
	\left. -30 H \left( \tfrac{n}{15} \right)+15 H \left( \tfrac{4n}{15} \right) \right).
	\end{multline*}
\end{lemma}
\begin{proof}
The equivalent formula in this case is
	\[
	\Theta_{{\bm{a}}}=\sum_{m\pmod{4}} c(m)\left(\mathcal{H}_{12,5}-3\mathcal{H}_{4,5}\big|V_3-2\mathcal{H}_{3,5} +6\mathcal{H}_{1,5}\big|V_{3}\right) \big|S_{4,m}. \qedhere
	\]
\end{proof}

\begin{lemma}\label{lem:2510}
For $\bm{a}=(2\ 5\ 10)^T$, we have 
\begin{equation*}
r_{{\bm a}}(n)=c(n)\left( H(100n) -4 H(25n) - 5H(4n) +20H(n)\right).
\end{equation*}
\end{lemma}
\begin{proof}
The equivalent formula in this case is
\[
\Theta_{{\bm{a}}}=-\sum_{m\pmod{4}} c(m)\mathcal{H}_{1,5}\big|\left(4-U_4\right)\big|U_5\big|S_{4,m}.\qedhere
\]
\end{proof}

\begin{lemma}\label{lem:2515}
 	For $\bm a=(2\ 5\ 15)^T$ we have 
 	\begin{multline*}
 	r_{{\bm a}}(n)=c(n) \left( H(600n) -2H(150n)- 5 H (24n) +10H(6n)+6 H\left(\tfrac{
50n
}{
3
}\right)-3 H \left( \tfrac{200n}{3} \right) 
 	\right.\\\left.
 	+15 H \left(\tfrac{8n}{3} \right)   
 	-30  H \left( \tfrac{2n}{3} \right) \right).
 	\end{multline*}
\end{lemma}
\begin{proof}
The equivalent formula in this case is
	\[
	\Theta_{{\bm{a}}}=\sum_{m\pmod{2}} c(m) \mathcal{H}_{1,5}\big|\left(U_{24}-3U_8V_3-2U_6+6U_2V_3\right)\big|U_5\big|S_{2,m}. \qedhere
\]

\end{proof}

\begin{lemma}\label{lem:269}
	We have, for $\bm a=(2\ 6\ 9)^T$,
	\[
	r_{{\bm a}}(n)= c(n)\left(H(12n)-2H(3n)-5H\left(\tfrac{4n}{3}\right)+10H\left(\tfrac{n}{3}\right) \right).
\]
\end{lemma}
\begin{proof}
The equivalent formula in this case is
	\[
	\Theta_{{\bm{a}}}=\sum_{m\pmod{12}} c(m) \mathcal{H}_{1,2}\big|\left(U_3-5V_3\right)\big|U_2\big|S_{12,m}. \qedhere
\]
\end{proof}

\begin{lemma}\label{lem:2615}
For $\bm{a}=(2\ 6\ 15)^T$, we have
\begin{multline*}
r_{{\bm a}, \bm h,N}(n)= c(n)\left( H(720n) -3 H(80n) -2 H(180n) + 6H (20n) +10 H \left( \tfrac{36n}{5} \right) 
-30
H \left( \tfrac{4n}{5} \right)  \right.\\
\left.  -5 H \left( \tfrac{144n}{5} \right) +15 H \left( \tfrac{16n}{5} \right)\right).
\end{multline*}
\end{lemma}
\begin{proof}
The equivalent formula in this case is
	\[
	\Theta_{{\bm{a}}}=\sum_{m\pmod{4}} 
	c(m)
	\left(\mathcal{H}_{80,3}-2\mathcal{H}_{20,3}+10\mathcal{H}_{4,3}\big|V_5-5\mathcal{H}_{16,3}\big|V_5\right) \big|U_3 \big|S_{4,m}. \qedhere
	\]
\end{proof}

\begin{lemma}\label{lem:334}
For $\bm{a}=(3\ 3\ 4)^T$, we have 
\begin{equation*}
r_{{\bm a}}(n)=
c(n)
 \left( H(36n) - 3 H(4n) + 8 H\left(\tfrac{9n}{4}\right) - 24H \left( \tfrac n4 \right) \right).
\end{equation*}

\end{lemma}
\begin{proof}
The equivalent formula in this case is
	\[
	\Theta_{{\bm{a}}}=\sum_{m\pmod{8}} 
	c(m)
	\left(\mathcal{H}_{4,3}+8\mathcal{H}_{1,3}\big|V_4\right) |U_3 \big|S_{8,m}. \qedhere
	\]
\end{proof}

\begin{lemma}\label{lem:337}
For $\bm{a}=(3\ 3\ 7)^T$, we have	
\begin{multline*}
r_{{\bm a}}(n)= H(252n) -2 H(63n) -3 H(28n) +6H(7n)-7 H \left( \tfrac{36n}{7} \right) +14 H \left( \tfrac{9n}{7} \right) \\
  +21 H \left( \tfrac{4n}{7} \right)-42  H \left(\tfrac{n}{7}\right).
\end{multline*}
\end{lemma}
\begin{proof}
The equivalent formula in this case is
	\[
	\Theta_{{\bm{a}}}=\mathcal{H}_{1,7}\big|\left(6-3U_4-2U_9+U_{36}\right). \qedhere 
	\]
\end{proof}

\begin{lemma}\label{lem:338}
For $\bm{a}=(3\ 3\ 8)^T$, we have 
\begin{equation*}
r_{{\bm a}}(n)=
c(n)
 \left( H(72n) -H(18n)- 
3
H(8n) +3H(2n)
-6\delta_{8\mid n}\left( H\left(\tfrac{9n}{2}\right)
-3 H 
\left( \tfrac n2 \right)\right)  \right).
\end{equation*}

\end{lemma}
\begin{proof}
The equivalent formula in this case is
	\[
	\Theta_{{\bm{a}}}= \sum_{m\pmod{8}}c(m)\left(\mathcal{H}_{8,3}-\mathcal{H}_{4,3}\big|V_2 -6\mathcal{H}_{2,3}\big|V_4\right)\big|U_3\big| S_{8,m}. \qedhere
	\]
\end{proof}

\begin{lemma}\label{lem:3412}
For $\bm{a}=(3\ 4\ 12)^T$, we have
\[
r_{{\bm a}}(n)=c(n)\left( H(9n) -3 H(n) +2H\left(\tfrac{9n}{4}\right)-6H\left(\tfrac{n}{4}\right) \right).
\]
\end{lemma}
\begin{proof}
The equivalent formula in this case is
	\[
	\Theta_{{\bm{a}}}= \sum_{m\pmod{8}}c(m)\mathcal{H}_{1,3}\big|\left(1+2V_4\right)\big|U_3 \big|S_{8,m}. \qedhere
	\]
\end{proof}

Setting $c_{(3\ 7\ 7)^T}(n):=1$, we have the following.
\begin{lemma}\label{lem:3763}
For $\bm{a}\in\{(3\ 7\ 7)^T,(3\ 7\ 63)^T\}$, we have
\begin{multline*}
r_{{\bm a}}(n)=c(n) \left( H(588n) -2H(147n)-7 H(12n) +14H(3n)-3 H \left(\tfrac{196n}{3} \right)+6H\left(\tfrac{49n}{3}\right) \right.\\
\left.
+21 H \left( \tfrac{4n}{3} \right)-42H\left(\tfrac{n}{3}\right)\right).
\end{multline*}
\end{lemma}
\begin{proof}
The equivalent formula in this case is 
\[
\Theta_{\bm{a}}=-\sum_{m\pmod{M}} 
c_{\bm{a}}(m)
\mathcal{H}_{1,7}\big|\left(2-U_4\right)\big|\left(U_3-3V_3\right)\big|U_7\big|S_{M,m}. \qedhere
\]
\end{proof}

\begin{lemma}\label{lem:3848}
	We have, for $\bm{a}=(3\ 8\ 48)^T$,
	\begin{equation*}
r_{{\bm a}}(n)=c(n)\left( 2H(72n) - 6 H(8n) -11 H \left( \tfrac{9n}{2} \right) +33 H \left( \tfrac n2 \right)\right).
\end{equation*}
\end{lemma}
\begin{proof}
The equivalent formula in this case is
	\[
	\Theta_{{\bm{a}}}=\sum_{m\pmod{32}} c(m)\mathcal{H}_{1,3}\big|\left(2U_8-11V_2\right)\big|U_{3}\big|S_{32,m}. \qedhere
	\]
\end{proof}

\begin{lemma}\label{lem:31030}
For $\bm{a}=(3\ 10\ 30)^T$, we have	
\begin{multline*}
r_{{\bm a}}(n)= c(n) \left( H(900n)-2 H \left( 225n \right)  -3H(100n) -5 H (36n)+6 H(25n)  +10 H(9n) \right. \\
\left. +15H(4n) -30 H(n)\right).
\end{multline*}
\end{lemma}
\begin{proof}
The equivalent formula in this case is
\[
	\Theta_{{\bm{a}}}=\sum_{m\pmod{4}} c(m)\mathcal{H}_{1,5}\big|\left(6-3U_4-2U_9+U_{36}\right)\big|U_5\big|S_{4,m}.\qedhere
	\]

\end{proof}

\begin{lemma}\label{lem:5615}
For $\bm{a}= (5\ 6\ 15)^T$, we have 
	\begin{multline*}
	r_{{\bm a}}(n)=c(n) \left( H(1800n)-2 H(450n)
-3H(200n)
 - 5H(72n) + 6 H(50n)  +10H(18n) 
 \right. \\ \left.
 +15 H(8n)-30H(2n)	\right).
	\end{multline*}
\end{lemma}
\begin{proof}
The equivalent formula  in this case is
\[
\Theta_{{\bm{a}}}=\sum_{m\pmod{2}} c(m)\mathcal{H}_{2,3}\big|\left(5-U_{25}\right)\big|\left(2-U_{4}\right)\big|U_3 \big|S_{2,m}. \qedhere
\]
\end{proof}

\begin{lemma}\label{lem:5840}
For $\bm{a}=(5\ 8\ 40)^T$, we have 
	\begin{equation*}
	r_{{\bm a}}(n)= c(n) \left(H\left(100n\right) - 5H\left(4n\right) - 10H\left(\tfrac{25n}{4}\right)  + 50H\left(\tfrac{n}{4}\right)\right).
	\end{equation*}
\end{lemma}
\begin{proof}
The equivalent formula in this case is
	\[
	\Theta_{{\bm{a}}}=\sum_{m\pmod{16}} c(m) \mathcal{H}_{1,5}\big|\left(U_4-10V_4\right)\big|U_5\big| S_{16,m}. \qedhere
	\]
\end{proof}

\newpage
\appendix

\section{Tables for the constants \texorpdfstring{$d_{\bm{a},\bm{h},N}(m)$}{d_{a,h,N}(m)}. }\label{sec:appendixds}
Here we list $d_{\bm{a},\bm{h},N}(m)$ from the statement of Theorem \ref{thm:congcondition}. For each $\bm{a}\in\mathcal{C}$, we give a table listing for $(\bm{h},N)\in \mathcal{S}_{\bm{a}}$ the corresponding $m$, $M$, and $d=d_{\bm{a},\bm{h},N}(n)$ valid for every $n\equiv m\pmod{M}$ such that $\Theta_{\bm{a},\bm{h},N} = \sum_{m\pmod{M}} d_{\bm{a},\bm{h},N}(m) \Theta_{\bm{a}}\big|S_{M,m}$. 

\begin{center}
\begin{tabular}{|c|c|c|c|c|| c ||c|c|c|c|c|| c ||c|c|c|c|c|}
\hline
\small{$\bm{h}$}&\small{$N$}& \small{$m$} &\small{$M$}& \small{$d$} & &\small{$\bm{h}$}&\small{$N$}& \small{$m$} &\small{$M$}&\small{$d$}& &\small{$\bm{h}$}&\small{$N$}& \small{$m$} &\small{$M$}&\small{$d$}\\
\hline\hline
\multicolumn{5}{|c||}{\small{$\bm{a}=(1\ 1\ 1)^T\vphantom{T^{T^T}}$}}&&\multicolumn{5}{c||}{\small{$\bm{a}=(1\ 1\ 3)^T$ (cont.)}}&&
\multicolumn{5}{c|}{\small $\bm{a}=(1\ 1\ 6)^T$ (cont.)}
\\[1pt]
\hline
\small{$(0\ 0\ 1)^{T}\vphantom{T^{T^T}}$} &\small{$3$} &\small{$1$} &\small{$3$} & \small{$\frac{1}{6}$} & &
\small{$(1\ 3\ 2)^T\vphantom{T^{T^T}}$} &\small{$8$}& \small{$6$} &\small{$8$}&\small{$\frac{1}{16}$}&&
\small{$(3\ 3\ 2)^T$} &\small{$6$} & \small{$114$}&\small{$144$}&\small{$\frac{1}{2}$}
\\[1pt]
\hline
 \small{$(0\ 1\ 1)^T\vphantom{T^{T^T}}$} &\small{$3$}& \small{$2$} &\small{$3$}&\small{$\frac{1}{12}$} &&
\small{$(3\ 3\ 2)^T$} &\small{$12$}& \small{$30$} &\small{$72$}&\small{$\frac{1}{8}$}&&
\small{$(2\ 2\ 1)^T$} &\small{$8$} & \small{$14$}&\small{$32$}&\small{$\frac{1}{8}$}
\\[1pt]
\hline
\small{$(0\ 1\ 1)^T\vphantom{T^{T^T}}$} &\small{$4$}&\small{$2$}&\small{$8$}&\small{$\frac{1}{12}$}&& 
 \small{$(3\ 3\ 4)^T$} &\small{$12$}&\small{$66$}&\small{$72$}&\small{$\frac{1}{8}$}&&
\small{$(2\ 2\ 3)^T$} &\small{$8$}& \small{$30$} &\small{$32$}&\small{$\frac{1}{8}$}
\\[1pt]
\hline
\small{$(0\ 1\ 2)^T\vphantom{T^{T^T}}$}  &\small{$4$}&\small{$5$} &\small{$8$}&\small{$\frac{1}{12}$}  & & 
\multicolumn{5}{c||}{\small{$\bm{a}=(1\ 1\ 4)^T$}}&&
\multicolumn{5}{c|}{\small $\bm{a}=(1\ 1\ 8)^T$}
\\[1pt]
\hline
\small{$(1\ 1\ 1)^T\vphantom{T^{T^T}}$} &\small{$4$} &\small{$3$}&\small{$8$}&\small{$\frac{1}{8}$}&& 
\small{$(0\ 0\ 1)^T$} &\small{$4$}&\small{$4$}&\small{$16$}&\small{$\frac{1}{6}$}&&
\small{$(0\ 0\ 1)^T$} &\small{$4$} & \small{$8$}&\small{$16$}&\small{$\frac{1}{6}$} 
\\
\hline
\end{tabular}
\end{center}

\begin{center}
\begin{tabular}{|c|c|c|c|c|| c ||c|c|c|c|c|| c ||c|c|c|c|c|}
\hline
			\small{$\bm{h}$}&\small{$N$}& \small{$m$} &\small{$M$}& \small{$d$} & &\small{$\bm{h}$}&\small{$N$}& \small{$m$} &\small{$M$}&\small{$d$} & &\small{$\bm{h}$}&\small{$N$}& \small{$m$} &\small{$M$}&\small{$d$}\\
			\hline\hline
\multicolumn{5}{|c||}{\small{$\bm{a}=(1\ 1\ 21)^T\vphantom{T^{T^T}}$ (cont.)}}&&
\multicolumn{5}{c||}{\small{$\bm{a}=(1\ 2\ 4)^T$}}&&
\multicolumn{5}{c|}{\small{$\bm{a}=(1\ 2\ 6)^T$ (cont.)}}
\\[1pt]
\hline
\small{$(0\ 0\ 3)^T\vphantom{T^{T^T}}$} &\small{$7$} & \small{$42$}&\small{$49$}&\small{$\frac{1}{2}$}&&
		\small{$(0\ 0\ 1)^T\vphantom{T^{T^T}}$} &\small{$4$} & \small{$4$}&\small{$16$}&\small{$\frac{1}{4}$}&& 
\small{$\vphantom{T^{T^T}} (0\ 2\ 1)^T$} &\small{$6$}& \small{$123$} &\small{$144$}&\small{$\frac{1}{6}$}
\\[1pt]
\hline
\small{$(7\ 7\ 1)^T\vphantom{T^{T^T}}$} &\small{$14$}& \small{$119$} &\small{$392$}&\small{$\frac{1}{2}$}&&
		\small{$(0\ 1\ 1)^T\vphantom{T^{T^T}}$} &\small{$4$}& \small{$6$} &\small{$16$}&\small{$\frac{1}{8}$}&&
\small{$(0\ 3\ 1)^T$} &\small{$6$} & \small{$24$}&\small{$144$}&\small{$\frac{1}{9}$}
\\[1pt]
\hline
\small{$(7\ 7\ 3)^T\vphantom{T^{T^T}}$} &\small{$14$} & \small{$287$}&\small{$392$}&\small{$\frac{1}{2}$}&&
		\small{$(0\ 2\ 1)^T\vphantom{T^{T^T}}$} &\small{$4$}& \small{$12$} &\small{$16$}&\small{$\frac{1}{4}$}&&
\small{$(0\ 3\ 1)^T$} &\small{$6$}& \small{$60$} &\small{$144$}&\small{$\frac{1}{9}$}
\\[1pt]
\hline
\small{$(7\ 7\ 5)^T\vphantom{T^{T^T}}$} &\small{$14$}& \small{$231$} &\small{$392$}&\small{$\frac{1}{2}$}&& 
		\small{$(1\ 1\ 1)^T\vphantom{T^{T^T}}$} &\small{$4$} & \small{$7$}&\small{$8$}&\small{$\frac{1}{8}$}&&
\small{$(3\ 0\ 1)^T$} &\small{$6$}& \small{$15$} &\small{$144$}&\small{$\frac{1}{6}$}
\\
\hline
\end{tabular}
\end{center}

\begin{center}
\begin{tabular}{|c|c|c|c|c|| c ||c|c|c|c|c|| c ||c|c|c|c|c|}
\hline
			\small{$\bm{h}$}&\small{$N$}& \small{$m$} &\small{$M$}& \small{$d$} & &\small{$\bm{h}$}&\small{$N$}& \small{$m$} &\small{$M$}&\small{$d$} & &\small{$\bm{h}$}&\small{$N$}& \small{$m$} &\small{$M$}&\small{$d$}\\
			\hline\hline
\multicolumn{5}{|c||}{\small{$\bm{a}=(1\ 2\ 16)^T\vphantom{T^{T^T}}$ (cont.)}}&&
\multicolumn{5}{c||}{\small{$\bm{a}=(1\ 3\ 6)^T$ (cont.)}}&&
\multicolumn{5}{c|}{\small{$\bm{a}=(1\ 3\ 10)^T$ (cont.)}}
\\[1pt]
\hline			
\small{$(0\ 4\ 3)^T\vphantom{T^{T^T}}$} &\small{$8$} & \small{$240$}&\small{$256$}&\small{$\frac{1}{4}$}&&
	\small{$\vphantom{T^{T^T}} (0\ 0\ 1)^T$} &\small{$6$}& \small{$42$} &\small{$72$}&\small{$\frac{1}{6}$}&& 
\small{$(5\ 5\ 2)^T$} &\small{$10$}& \small{$140$} &\small{$200$}&\small{$\frac{1}{3}$}
\\[1pt]
\hline
\multicolumn{5}{|c||}{\small{$\bm{a}=(1\ 3\ 3)^T$}}&&
 \small{$\vphantom{T^{T^T}} (0\ 0\ 1)^T$} &\small{$6$} & \small{$6$}&\small{$144$}&\small{$\frac{1}{2}$} &&
\small{$(5\ 5\ 3)^T$} &\small{$10$} & \small{$190$}&\small{$200$}&\small{$\frac{1}{3}$}
\\[1pt]
\hline
			\small{$\vphantom{T^{T^T}} (0\ 0\ 1)^T$} &\small{$3$}&\small{$3$}&\small{$9$}&\small{$\frac{1}{4}$}& & 
			\small{$(0\ 2\ 3)^T$} &\small{$6$}&\small{$66$}&\small{$72$}&\small{$\frac{1}{6}$}&&
			\small{$(5\ 5\ 4)^T$} &\small{$10$}& \small{$60$} &\small{$200$}&\small{$\frac{1}{3}$}
\\[1pt]
\hline
			\small{$\vphantom{T^{T^T}} (0\ 1\ 1)^T$} &\small{$3$}& \small{$6$} &\small{$9$}&\small{$\frac{1}{4}$}&&
			\small{$(0\ 2\ 3)^T$} &\small{$6$}& \small{$102$} &\small{$144$}&\small{$\frac{1}{1}$}&&
			\small{$(0\ 2\ 3)^T$} &\small{$12$} & \small{$102$}&\small{$144$}&\small{$\frac{1}{12}$}
\\[1pt]
\hline
			\small{$\vphantom{T^{T^T}} (0\ 1\ 1)^T$} &\small{$4$} & \small{$6$}&\small{$8$}&\small{$\frac{1}{4}$} &&
			\small{$(3\ 1\ 0)^T$} &\small{$6$}& \small{$12$} &\small{$72$}&\small{$\frac{1}{3}$}&&
			\small{$(6\ 4\ 3)^T$} &\small{$12$} & \small{$30$}&\small{$144$}&\small{$\frac{1}{12}$}
\\
\hline
\end{tabular}
\end{center}

\begin{center}
\begin{tabular}{|c|c|c|c|c|| c ||c|c|c|c|c|| c ||c|c|c|c|c|}
			\hline
			\small{$\bm{h}$}&\small{$N$}& \small{$m$} &\small{$M$}& \small{$d$} & &\small{$\bm{h}$}&\small{$N$}& \small{$m$} &\small{$M$}&\small{$d$} & &\small{$\bm{h}$}&\small{$N$}& \small{$m$} &\small{$M$}&\small{$d$}\\
			\hline\hline
\multicolumn{5}{|c||}{\small{$\bm{a}=(1\ 3\ 30)^T\vphantom{T^{T^T}}$ (cont.)}}&&
\multicolumn{5}{c||}{\small{$\bm{a}=(1\ 4\ 8)^T$}}&&
\multicolumn{5}{c|}{\small{$\bm{a}=(1\ 4\ 24)^T\vphantom{T^{T^T}}$ (cont.)}}
\\[1pt]
\hline
			\small{$(5\ 5\ 3)^T\vphantom{T^{T^T}}$} &\small{$10$}& \small{$170$} &\small{$200$}&\small{$\frac{1}{3}$}&&
		\small{$\vphantom{T^{T^T}} (0\ 0\ 1)^T$} &\small{$4$}& \small{$24$} &\small{$64$}&\small{$\frac{1}{4}$}&&
	\small{$(0\ 3\ 2)^T$} &\small{$6$} & \small{$168$}&\small{$576$}&\small{$\frac{1}{4}$}			
\\[1pt]
\hline
			\small{$(5\ 5\ 4)^T\vphantom{T^{T^T}}$} &\small{$10$}& \small{$180$} &\small{$200$}&\small{$\frac{1}{3}$}&&
	\small{$\vphantom{T^{T^T}} (0\ 0\ 1)^T$} &\small{$4$} & \small{$8$}&\small{$64$}&\small{$\frac{1}{6}$}&&
			\small{$(0\ 3\ 2)^T$} &\small{$6$}& \small{$456$} &\small{$576$}&\small{$\frac{1}{2}$}			
\\[1pt]
\hline
\multicolumn{5}{|c||}{\small{$\bm{a}=(1\ 4\ 4)^T$}}&&
			\small{$\vphantom{T^{T^T}} (0\ 1\ 1)^T$} &\small{$4$}& \small{$12$} &\small{$16$}&\small{$\frac{1}{8}$}&& 
			\small{$(4\ 2\ 1)^T$} &\small{$8$} & \small{$65$}&\small{$128$}&\small{$\frac{1}{4}$} 			
\\[1pt]
\hline
			\small{$\vphantom{T^{T^T}} (0\ 0\ 1)^T$} &\small{$4$}& \small{$8$} &\small{$16$}&\small{$\frac{1}{12}$}&& 
	\small{$\vphantom{T^{T^T}} (0\ 2\ 1)^T$} &\small{$4$}& \small{$24$} &\small{$64$}&\small{$\frac{1}{4}$}&& 
			\small{$(4\ 2\ 3)^T$} &\small{$8$} & \small{$120$}&\small{$128$}&\small{$\frac{1}{4}$}
\\
\hline 
		\end{tabular}
\end{center}

\begin{center}
\begin{tabular}{|c|c|c|c|c|| c ||c|c|c|c|c|| c ||c|c|c|c|c|}
\hline
\small{$\bm{h}$}&\small{$N$}& \small{$m$} &\small{$M$}&\small{$d$} &&
\small{$\bm{h}$}&\small{$N$}& \small{$m$} &\small{$M$}&\small{$d$}&&
\small{$\bm{h}$}&\small{$N$}& \small{$m$} &\small{$M$}&\small{$d$}
 \\
\hline\hline
\multicolumn{5}{|c||}{\small{$\bm{a}=(1\ 5\ 40)^T\vphantom{T^{T^T}} $ (cont.)}}&&
\multicolumn{5}{c||}{\small{$\bm{a}=(1\ 6\ 9)^T\vphantom{T^{T^T}}$ (cont.)}}&&
\multicolumn{5}{c||}{\small{$\bm{a}=(1\ 8\ 8)^T$ (cont.)}}
 \\[1pt]
\hline
			\small{$\vphantom{T^{T^T}} (4\ 4\ 3)^T$} &\small{$8$}&\small{$72$}&\small{$128$}&\small{$\frac{1}{2}$} &&
			\small{$\vphantom{T^{T^T}} (6\ 3\ 2)^T$} &\small{$12$}& \small{$414$} &\small{$432$}&\small{$\frac{1}{8}$}&&
			\small{$\vphantom{T^{T^T}} (0\ 1\ 1)^T$} &\small{$4$} & \small{$16$}&\small{$64$}&\small{$\frac{1}{6}$}
 \\[1pt]
\hline
\multicolumn{5}{|c||}{\small{$\bm{a}=(1\ 6\ 6)^T\vphantom{T^{T^T}} $}}&&
\multicolumn{5}{c||}{\small{$\bm{a}=(1\ 6\ 16)^T$}}&&
			\small{$\vphantom{T^{T^T}} (0\ 1\ 1)^T$} &\small{$4$} & \small{$32$}&\small{$64$}&\small{$\frac{1}{6}$}
\\[1pt]
\hline
			\small{$\vphantom{T^{T^T}} (0\ 0\ 1)^T$} &\small{$3$}& \small{$6$} &\small{$9$}&\small{$\frac{1}{4}$}&& 
	\small{$\vphantom{T^{T^T}} (0\ 1\ 0)^T$} &\small{$3$}& \small{$6$} &\small{$9$}&\small{$\frac{1}{4}$}&& 
			\small{$\vphantom{T^{T^T}} (0\ 1\ 2)^T$} &\small{$4$} & \small{$40$}&\small{$64$}&\small{$\frac{1}{4}$}
\\[1pt]
\hline
			\small{$\vphantom{T^{T^T}} (0\ 1\ 1)^T$} &\small{$3$} & \small{$3$}&\small{$9$}&\small{$\frac{1}{4}$}&& 
			\small{$\vphantom{T^{T^T}} (0\ 2\ 1)^T$} &\small{$4$} & \small{$40$}&\small{$64$}&\small{$\frac{1}{4}$}&& 
			\small{$\vphantom{T^{T^T}} (0\ 1\ 2)^T$} &\small{$4$} & \small{$56$}&\small{$64$}&\small{$\frac{1}{4}$}
\\
\hline
		\end{tabular}
\end{center}

\begin{center}
\begin{tabular}{|c|c|c|c|c|| c ||c|c|c|c|c|| c ||c|c|c|c|c|}
			\hline
			\small{$\bm{h}$}&\small{$N$}& \small{$m$} &\small{$M$}& \small{$d$} & &\small{$\bm{h}$}&\small{$N$}& \small{$m$} &\small{$M$}&\small{$d$} & &\small{$\bm{h}$}&\small{$N$}& \small{$m$} &\small{$M$}&\small{$d$}\\
\hline\hline
			
\multicolumn{5}{|c||}{\small{$\bm{a}=(1\ 9\ 12)^T$ (cont.)}}&&
\multicolumn{5}{c||}{\small{$\bm{a}=(1\ 9\ 24)^T$ (cont.)}}&&
\multicolumn{5}{c|}{\small{$\bm{a}=(1\ 16\ 24)^T$ (cont.)}}	
\\[1pt]
\hline
			\small{$\vphantom{T^{T^T}} (3\ 1\ 0)^T$} &\small{$6$}& \small{$18$} &\small{$216$}&\small{$\frac{1}{2}$}&&  
			\small{$\vphantom{T^{T^T}} (0\ 4\ 3)^T$} &\small{$12$} & \small{$936$}&\small{$1728$}&\small{$\frac{1}{8}$}&&
			\small{$\vphantom{T^{T^T}} (0\ 3\ 2)^T$} &\small{$8$} & \small{$240$}&\small{$256$}&\small{$\frac{1}{8}$}
\\[1pt]
\hline
			\small{$\vphantom{T^{T^T}} (3\ 1\ 0)^T$} &\small{$6$} & \small{$90$}&\small{$216$}&\small{$\frac{1}{4}$}&& 
			\small{$\vphantom{T^{T^T}} (12\ 4\ 3)^T$} &\small{$24$} & \small{$504$}&\small{$3456$}&\small{$\frac{1}{4}$}&&			
			\small{$\vphantom{T^{T^T}} (0\ 3\ 4)^T$} &\small{$12$} & \small{$528$}&\small{$576$}&\small{$\frac{1}{8}$}
\\[1pt]
\hline
			\small{$\vphantom{T^{T^T}} (3\ 1\ 3)^T$} &\small{$6$} & \small{$126$}&\small{$216$}&\small{$\frac{1}{2}$}&& 
			\small{$\vphantom{T^{T^T}} (12\ 4\ 3)^T$} &\small{$24$} & \small{$1656$}&\small{$3456$}&\small{$\frac{1}{8}$}&&			
			\small{$\vphantom{T^{T^T}} (0\ 3\ 4)^T$} &\small{$12$} & \small{$672$}&\small{$2304$}&\small{$\tfrac{1}{8}$}
\\[1pt]
\hline
			\small{$\vphantom{T^{T^T}} (3\ 1\ 3)^T$} &\small{$6$} & \small{$196$}&\small{$216$}&\small{$\frac{1}{4}$}&&
			\small{$\vphantom{T^{T^T}} (12\ 4\ 9)^T$} &\small{$24$} & \small{$2232$}&\small{$3456$}&\small{$\frac{1}{4}$}&&			
			\small{$\vphantom{T^{T^T}} (0\ 3\ 4)^T$} &\small{$12$} & \small{$1824$}&\small{$2304$}&\small{$\tfrac{1}{4}$}
\\
\hline
		\end{tabular}
\end{center}

\begin{center}
\begin{tabular}{|c|c|c|c|c|| c ||c|c|c|c|c|| c ||c|c|c|c|c|}
			\hline
			\small{$\bm{h}$}&\small{$N$}& \small{$m$} &\small{$M$}& \small{$d$} & &\small{$\bm{h}$}&\small{$N$}& \small{$m$} &\small{$M$}&\small{$d$} & &\small{$\bm{h}$}&\small{$N$}& \small{$m$} &\small{$M$}&\small{$d$}\\
			\hline\hline
\multicolumn{5}{|c||}{\small{$\vphantom{T^{T^T}} \bm{a}=(2\ 2\ 3)^T$ (cont.)}}&&
\multicolumn{5}{c||}{\small{$\bm{a}=(2\ 3\ 6)^T$}}&&
\multicolumn{5}{c|}{\small{$\bm{a}=(2\ 3\ 8)^T$ (cont.)}}
\\[1pt]
\hline			
			\small{$\vphantom{T^{T^T}} (0\ 3\ 1)^T$} &\small{$6$}& \small{$2$} &\small{$72$}&\small{$\frac{1}{4}$}&&
			\small{$\vphantom{T^{T^T}} (0\ 0\ 1)^T$} &\small{$3$}& \small{$6$} &\small{$9$}&\small{$\frac{1}{2}$}&&
			\small{$\vphantom{T^{T^T}} (12\ 4\ 3)^T$} &\small{$24$} & \small{$408$}&\small{$1152$}&\small{$\frac{1}{8}$}
\\[1pt]
\hline
			\small{$\vphantom{T^{T^T}} (0\ 3\ 2)^T$} &\small{$6$}& \small{$30$} &\small{$36$}&\small{$\frac{1}{4}$}&&
			\small{$\vphantom{T^{T^T}} (0\ 1\ 0)^T$} &\small{$3$}& \small{$3$} &\small{$9$}&\small{$\frac{1}{2}$}&&
			\small{$\vphantom{T^{T^T}} (12\ 4\ 9)^T$} &\small{$24$} & \small{$984$}&\small{$1152$}&\small{$\frac{1}{8}$}
\\[1pt]
\hline
				\small{$\vphantom{T^{T^T}} (3\ 3\ 1)^T$} &\small{$6$}& \small{$39$} &\small{$72$}&\small{$\frac{1}{2}$}&&
			\small{$\vphantom{T^{T^T}} (1\ 0\ 1)^T$} &\small{$4$}& \small{$8$} &\small{$16$}&\small{$\frac{1}{6}$}&&
\multicolumn{5}{c|}{\small{$\bm{a}=(2\ 3\ 9)^T$}}
\\[1pt]
\hline
				\small{$\vphantom{T^{T^T}} (1\ 2\ 2)^T$} &\small{$8$}& \small{$22$} &\small{$32$}&\small{$\frac{1}{16}$}&&
			\small{$\vphantom{T^{T^T}} (1\ 2\ 1)^T$} &\small{$4$}& \small{$4$} &\small{$16$}&\small{$\frac{1}{6}$}&&
			\small{$\vphantom{T^{T^T}} (1\ 0\ 2)^T$} &\small{$4$}& \small{$6$} &\small{$16$}&\small{$\frac{1}{6}$}
\\
\hline
		\end{tabular}
\end{center}

\begin{center}
\begin{tabular}{|c|c|c|c|c|| c ||c|c|c|c|c|| c ||c|c|c|c|c|}
			\hline
			\small{$\bm{h}$}&\small{$N$}& \small{$m$} &\small{$M$}& \small{$d$} & &\small{$\bm{h}$}&\small{$N$}& \small{$m$} &\small{$M$}&\small{$d$} & &\small{$\bm{h}$}&\small{$N$}& \small{$m$} &\small{$M$}&\small{$d$}\\
\hline\hline
\multicolumn{5}{|c||}{\small{$\bm{a}=(2\ 3\ 18)^T$ (cont.)}}&&
\multicolumn{5}{|c||}{\small{$\bm{a}=(2\ 5\ 6)^T$ (cont.)}} &&
\multicolumn{5}{c||}{\small{$\bm{a}=(2\ 6\ 15)^T$ (cont.)}}			
\\[1pt]
\hline
			\small{$\vphantom{T^{T^T}} (0\ 3\ 2)^T$} &\small{$6$}& \small{$99$} &\small{$216$}&\small{$\frac{1}{4}$}&&	
			\small{$\vphantom{T^{T^T}} (5\ 8\ 5)^T$} &\small{$20$}& \small{$120$} &\small{$400$}&\small{$\frac{1}{12}$}&&	
			\small{$\vphantom{T^{T^T}} (5\ 5\ 4)^T$} &\small{$20$}& \small{$40$} &\small{$400$}&\small{$\frac{1}{12}$}
\\[1pt]
\hline
			\small{$\vphantom{T^{T^T}} (0\ 3\ 2)^T$} &\small{$6$}& \small{$171$} &\small{$216$}&\small{$\frac{1}{2}$}&&
\multicolumn{5}{c||}{\small{$\bm{a}=(2\ 5\ 10)^T$}}&&			
			\small{$\vphantom{T^{T^T}} (5\ 5\ 6)^T$} &\small{$20$}& \small{$340$} &\small{$400$}&\small{$\frac{1}{12}$}
\\[1pt]
\hline
		\small{$\vphantom{T^{T^T}} (3\ 3\ 1)^T$} &\small{$6$}& \small{$63$} &\small{$216$}&\small{$\frac{1}{2}$}&&
			\small{$\vphantom{T^{T^T}} (1\ 0\ 1)^T$} &\small{$4$}& \small{$12$}&\small{$32$}&\small{$\frac{1}{4}$}&&
			\small{$\vphantom{T^{T^T}} (5\ 5\ 8)^T$} &\small{$20$}& \small{$360$} &\small{$400$}&\small{$\frac{1}{12}$}
\\[1pt]
\hline
\small{$\vphantom{T^{T^T}} (3\ 3\ 1)^T$} &\small{$6$}& \small{$207$} &\small{$216$}&\small{$\frac{1}{4}$}&&	
			\small{$\vphantom{T^{T^T}} (1\ 0\ 1)^T$} &\small{$4$}& \small{$28$} &\small{$32$}&\small{$\frac{1}{8}$}&&	
\multicolumn{5}{c|}{\small{$\bm{a}=(3\ 3\ 4)^T$}}
\\
\hline

		\end{tabular}
\end{center}

\begin{center}
\begin{tabular}{|c|c|c|c|c|| c ||c|c|c|c|c|| c ||c|c|c|c|c|}
			\hline
\small{$\bm{h}$}&\small{$N$}& \small{$m$} &\small{$M$}& \small{$d$} & &\small{$\bm{h}$}&\small{$N$}& \small{$m$} &\small{$M$}&\small{$d$}& &\small{$\bm{h}$}&\small{$N$}& \small{$m$} &\small{$M$}&\small{$d$}\\
\hline\hline
\multicolumn{5}{|c||}{\small{$\vphantom{T^{T^T}}\bm{a}=(3\ 3\ 7)^T$ (cont.)}}&&
\multicolumn{5}{c|}{\small{$\bm{a}=(3\ 4\ 4)^T$ (cont.)}}&&
\multicolumn{5}{c|}{\small{$\bm{a}=(3\ 7\ 7)^T$ (cont.)}}			
\\[1pt]
\hline
\small{$\vphantom{T^{T^T}}(7\ 7\ 9)^T$} &\small{$21$} & \small{$420$}& \small{$441$}&\small{$\frac{1}{8}$}&&
\small{$\vphantom{T^{T^T}}(8\ 3\ 9)^T$} &\small{$24$} & \small{$552$}& \small{$576$}&\small{$\frac{1}{16}$}&&
			\small{$\vphantom{T^{T^T}} (7\ 1\ 3)^T$} &\small{$14$}& \small{$21$} &\small{$196$}&\small{$\frac{1}{8}$}
\\[1pt]
\hline
\multicolumn{5}{|c||}{\small{$\bm{a}=(3\ 3\ 8)^T$}}&&
\multicolumn{5}{c||}{\small{$\bm{a}=(3\ 4\ 12)^T$}}&&
			\small{$\vphantom{T^{T^T}} (7\ 1\ 5)^T$} &\small{$14$}& \small{$133$} &\small{$196$}&\small{$\frac{1}{8}$}
\\[1pt]
\hline
			\small{$\vphantom{T^{T^T}} (0\ 1\ 0)^T$} &\small{$3$}& \small{$3$} &\small{$9$}&\small{$\frac{1}{4}$} &&
			\small{$\vphantom{T^{T^T}} (1\ 0\ 1)^T$} &\small{$3$}& \small{$6$} &\small{$9$}&\small{$\frac{1}{4}$} &&
\small{$(7\ 3\ 5)^T$} &\small{$14$}& \small{$189$} &\small{$196$}&\small{$\frac{1}{8}$}
\\[1pt]
\hline
			\small{$\vphantom{T^{T^T}} (1\ 1\ 0)^T$} &\small{$3$}&\small{$6$}&\small{$9$}&\small{$\frac{1}{4}$}&&
			\small{$\vphantom{T^{T^T}} (2\ 0\ 1)^T$} &\small{$4$}&\small{$24$}&\small{$32$}&\small{$\frac{1}{2}$}&&
			\small{$(7\ 3\ 6)^T$} &\small{$21$}& \small{$21$} &\small{$441$}&\small{$\frac{1}{16}$}
\\
\hline

		\end{tabular}
\end{center}

\begin{center}
\begin{tabular}{|c|c|c|c|c|| c ||c|c|c|c|c|| c ||c|c|c|c|c|}
			\hline
			\small{$\bm{h}$}&\small{$N$}& \small{$m$} &\small{$M$}& \small{$d$} & &\small{$\bm{h}$}&\small{$N$}& \small{$m$} &\small{$M$}&\small{$d$}& &\small{$\bm{h}$}&\small{$N$}& \small{$m$} &\small{$M$}&\small{$d$}
\\
\hline\hline
\multicolumn{5}{|c||}{\small{$\bm{a}=(3\ 8\ 12)^T$} (cont.)}&&
\multicolumn{5}{c||}{\small{$\bm{a}=(3\ 8\ 48)^T (cont.)$}}&&
\multicolumn{5}{c||}{\small{$\bm{a}=(3\ 10\ 30)^T$}}
\\[1pt]
\hline
			\small{$\vphantom{T^{T^T}} (2\ 1\ 2)^T$} &\small{$4$}&\small{$4$}&\small{$32$}&\small{$\frac{1}{4}$}&&
\small{$\vphantom{T^{T^T}} (0\ 0\ 1)^T$} &\small{$4$}& \small{$224$} &\small{$256$}&\small{$\frac{1}{4}$} &&
\small{$\vphantom{T^{T^T}} (1\ 0\ 1)^T$} &\small{$3$}& \small{$6$} &\small{$9$}&\small{$\frac{1}{4}$} 
\\[1pt]
\hline
			\small{$\vphantom{T^{T^T}} (2\ 0\ 1)^T$} &\small{$6$}&\small{$24$}&\small{$576$}&\small{$\frac{1}{4}$}&&
\small{$\vphantom{T^{T^T}} (0\ 0\ 1)^T$} &\small{$4$}& \small{$96$} &\small{$256$}&\small{$\frac{1}{2}$} &&
\small{$\vphantom{T^{T^T}} (0\ 1\ 1)^T$} &\small{$4$}&\small{$4$}&\small{$16$}&\small{$\frac{1}{6}$}
\\[1pt]
\hline
			\small{$\vphantom{T^{T^T}} (2\ 0\ 1)^T$} &\small{$6$}& \small{$60$} &\small{$144$}&\small{$\frac{1}{8}$}&&
\small{$\vphantom{T^{T^T}} (0\ 1\ 1)^T$} &\small{$4$}& \small{$40$} &\small{$64$}&\small{$\frac{1}{4}$} &&
\small{$\vphantom{T^{T^T}} (2\ 1\ 1)^T$} &\small{$4$} & \small{$4$}&\small{$16$}&\small{$\frac{1}{6}$}
\\[1pt]
\hline
			\small{$\vphantom{T^{T^T}} (2\ 0\ 1)^T$} &\small{$6$}& \small{$312$} &\small{$576$}&\small{$\frac{1}{8}$}&&
\small{$\vphantom{T^{T^T}} (0\ 1\ 1)^T$} &\small{$4$} & \small{$56$}&\small{$64$}&\small{$\frac{1}{8}$} &&
\small{$\vphantom{T^{T^T}} (1\ 0\ 1)^T$} &\small{$6$}& \small{$33$} &\small{$72$}&\small{$\frac{1}{4}$}
\\
\end{tabular}
\end{center}

\rm
\section{Tables for the constants \texorpdfstring{$c_{\bm{a}}(n)$}{c_{a}(n)}}\label{sec:appendixcs}

In this appendix, we give the constants $c_{\bm{a}}(n)$ occurring in the lemmas in Section \ref{sec:lemmas} for $n\equiv m\pmod{M}$. For this, define  $\varepsilon_{m,d}:=(-1)^{\delta_{d\mid m}}$. 
\begin{center}
\begin{tabular}{|c|c|c|c||c|c|c|c||c|c|c|c|} 
\hline
\small{$\bm{a}$}&\small{$m$} &\small{$M$}&\small{$c$} &
\small{$\bm{a}$}&\small{$m$} &\small{$M$}&\small{$c$} &\small  
\small{$\bm{a}$}&\small{$m$} &\small{$M$}&\small{$c$}
\\[1pt]
\hline\hline

\multirow{2}{*}{\small{$(1\ 1\ 2)^T$}} &\small{$0$} &\small{$2$}&\small{$12$}&
\small{$\vphantom{T^{T^T}}(1\ 1\ 12)^T$} &\small{$0$} &\small{$4$}  &\small{$\vphantom{T^{T^T}}\tfrac 23$}&
\multirow{2}{*}{\small{$(1\ 2\ 4)^T$}} &\small{$2,4$} &\small{$8$}  &\small{$4$}
\\[1pt]
\cline{2-4}\cline{5-8}\cline{10-12}

 &\small{$1$}&\small{$2$}&\small{$4$}&
\multirow{3}{*}{\small{$(1\ 1\ 21)^T$}}  &\small{$1,2$} &\small{$4$}  &\small{$1$} &
 &\small{$0,6$} &\small{$8$}  &\small{$12$} 
\\[1pt]
\cline{1-4}\cline{6-8}\cline{9-12}

\multirow{3}{*}{\small{$(1\ 1\ 4)^T$}}  &\small{$0$} &\small{$4$}  &\small{$12$}&
 &\small{$0$} &\small{$4$}  &\small{$\vphantom{T^{T^T}}\tfrac 13$}&
\multirow{2}{*}{\small{$(1\ 2\ 6)^T$}}  &\small{$2,3$} &\small{$4$}  &\small{$1$} 
\\[1pt]
\cline{2-4}\cline{6-8}\cline{10-12}

 &\small{$1$} &\small{$4$}  &\small{$8$}&
 &\small{$7$} &\small{$8$}  &\small{$\vphantom{T^{T^T}}\tfrac 12$}&
 &\small{$0,1$} &\small{$4$}  &\small{$3$}
\\[1pt]
\cline{2-4}\cline{5-8}\cline{9-12}

&\small{$2$} &\small{$4$}  &\small{$4$}&
\multirow{3}{*}{\small{$(1\ 1\ 24)^T$}}  &\small{$1,2,5$} &\small{$8$}  &\small{$2$}&
\multirow{3}{*}{\small{$(1\ 2\ 8)^T$}}  &\small{$\begin{array}{c}1,4,\\ 8,9\end{array}$} &\small{$16$}  &\small{$4$} 
\\[1pt]
\cline{1-4}\cline{6-8}\cline{10-12}

\multirow{3}{*}{\small{$(1\ 1\ 8)^T$}} &\small{$1$}&\small{$4$}&\small{$4$}&
&\small{$4$} &\small{$8$}  &\small{$1$}&
  &\small{$3$} &\small{$8$}  &\small{$6$}
\\[1pt]
\cline{2-4}\cline{6-8}\cline{10-12}

&\small{$2$}&\small{$8$}&\small{$8$}&
&\small{$0$} &\small{$8$}  &\small{$\vphantom{T^{T^T}}\tfrac{1}{2}$}&
 &\small{$0,12$} &\small{$16$}  &\small{$12$}

\\
\end{tabular}
\end{center}
\begin{center}
		\begin{tabular}{|c|c|c|c||c|c|c|c||c|c|c|c|}
\hline
\small{$\bm{a}$}&\small{$m$} &\small{$M$}&\small{$c$} &
\small{$\bm{a}$}&\small{$m$} &\small{$M$}&\small{$c$} &\small  
\small{$\bm{a}$}&\small{$m$} &\small{$M$}&\small{$c$}
\\[1pt]
\hline\hline

\multirow{2}{*}{\small{$\vphantom{T^{T^T}}(1\ 2\ 16)^T$}} &\small{$2,8$}&\multirow{2}{*}{\small{$32$}}&\multirow{2}{*}{\small{$4$}}&
\small{$\vphantom{T^{T^T}}(1\ 6\ 9)^T$}&\small{$0,3$} &\small{$6$}  &\small{$2\varepsilon_{m,2}$}&
\multirow{4}{*}{\small{$\vphantom{T^{T^T}}(1\ 16\ 24)^T$}}  &\small{$1$} &\small{$8$}  &\small{$1$}
\\[1pt]
\cline{5-8}\cline{10-12}
&\small{$16,18$}&&&
\multirow{4}{*}{\small{$(1\ 6\ 16)^T$}}  &\small{$\pm 1,\pm6$} &\small{$16$}  &\small{$1$}&
  &\small{$0,8$} &\small{$32$}  &\small{$-\tfrac12\vphantom{T^{T^T}}$}
\\[1pt]
\cline{1-4}\cline{6-8}\cline{10-12}
\multirow{3}{*}{\small{$(1\ 3\ 4)^T$}}  &\small{$1$} &\small{$8$}  &\small{$\vphantom{T^{T^T}}\tfrac 32$}&
 &\small{$\pm 4,\pm 12$}&\multirow{2}{*}{\small{$32$}}  &\multirow{2}{*}{\small{$\tfrac12$}} &
&\small{$\pm4, \pm12,$} &\multirow{2}{*}{\small{$32$}}  &\multirow{2}{*}{\small{$1$}}
\\[1pt]
\cline{2-4}
  &\small{$5$} &\small{$8$}  &\small{$1$}&
 &\small{$16,24$}&&&
&\small$16, 24$&&

\\[1pt]
\cline{2-4}\cline{6-8}\cline{9-12}
  &\small{$0,3$} &\small{$4$}  &\small{$2$}&
  &\small{$0,8$} &\small{$32$}  &\small{$-\tfrac12\vphantom{T^{T^T}}$}&
\multirow{3}{*}{\small{$\vphantom{T^{T^T}}(1\ 24\ 24)^T$}}  &\small{$1$} &\small{$8$}  &\small{$1$} 
\\[1pt]
\cline{1-4}\cline{5-8}\cline{10-12}
\multirow{2}{*}{\small{$(1\ 3\ 6)^T$}}  &\small{$1$} &\small{$2$}  &\small{$1$} &
\multirow{4}{*}{\small{$(1\ 6\ 18)^T$}}  &\small{$3,6,$} &\multirow{2}{*}{\small{$12$}}  &\multirow{2}{*}{\small{$\vphantom{T^{T^T}}\varepsilon_{m,3}$}}& 
&\small{$4,8$} &\small{$16$}  &\small{$2$} 
\\[1pt]
\cline{2-4}\cline{10-12}
  &\small{$0$} &\small{$2$}  &\small{$3$}&
&\small{$7,10$}&&&
&\small{$0,12$} &\small{$16$}  &\small{$-2$} 
\\
\end{tabular}
\end{center}

\begin{center}
		\begin{tabular}{|c|c|c|c||c|c|c|c||c|c|c|c|}
			\hline
			\small{$\bm{a}$}&\small{$m$} &\small{$M$}&\small{$c$}&			\small{$\bm{a}$}&\small{$m$} &\small{$M$}&\small{$c$}&		
\small{$\bm{a}$}&\small{$m$} &\small{$M$}&\small{$c$} \\
			\hline\hline
\small{$\vphantom{T^{T^T}}(2\ 6\ 9)^T$} &\small{$6,11$} &\small{$12$}  &\small{$\varepsilon_{m,3}$}&			
\multirow{2}{*}{\small{$(3\ 4\ 36)^T$}}  &\small{$0,3$} &\small{$12$}  &\small{$2 \varepsilon_{m,2}$}&
\multirow{6}{*}{\small{$(3\ 8\ 72)^T$}}  &\small{$3$} &\small{$24$}  &\small{$\tfrac 13\vphantom{T^{T^T}}$}
\\[1pt]
\cline{1-4}
\multirow{2}{*}{\small{$(2\ 6\ 15)^T$}}  &\small{$1,2$} &\small{$4$}  &\small{$\tfrac12\vphantom{T^{T^T}}$} &
 &\small{$4,7$} &\small{$12$}  &\small{$\varepsilon_{m,2}$}&
&\small{$8,44$} &\small{$48$}  &\small{$\tfrac 13\vphantom{T^{T^T}}$}
\\[1pt]
\cline{2-4}\cline{5-8}\cline{10-12}
						
 &\small{$0,3$} &\small{$4$}  &\small{$\tfrac32\vphantom{T^{T^T}}$}& 
\multirow{2}{*}{\small{$(3\ 7\ 63)^T$}}  &\small{$1$} &\small{$3$}  &\small{$\tfrac{1}{2}\vphantom{T^{T^T}}$}&
  &\small{$6,35$} &\small{$48$}  &\small{$\tfrac 16\vphantom{T^{T^T}}$}
			\\[1pt]
			\cline{1-4}\cline{6-8}\cline{10-12}
\multirow{3}{*}{\small{$(3\ 3\ 4)^T$}}  &\small{$2,7$} &\small{$8$}  &\small{$2$} &
 &\small{$0$} &\small{$3$}  &\small{$1$}&
 &\small{$12,24$} &\small{$48$}  &\small{$\tfrac 23\vphantom{T^{T^T}}$}
			\\[1pt]
			\cline{2-4}\cline{5-8}\cline{10-12}
 &\small{$3$} &\small{$8$}  &\small{$3$} &
			\multirow{3}{*}{\small{$\vphantom{T^{T^T}}(3\ 8\ 8)^T$}}  &\small{$3$} &\small{$8$}  &\small{$1$} &
&\small{$20,32$} &\small{$48$}  &\small{$-\tfrac 13\vphantom{T^{T^T}}$}
			\\[1pt]
			\cline{2-4}\cline{6-8}\cline{10-12}
			
 &\small{$0$} &\small{$4$}  &\small{$\tfrac23\vphantom{T^{T^T}}$} &
 &\small{$0,4$} &\small{$16$}  &\small{$-2$} & 
&\small{$0,36$} &\small{$48$}  &\small{$-\tfrac 23\vphantom{T^{T^T}}$}
			\\
\end{tabular}
\end{center}

\section{Bounds for the valence formula}\label{sec:appendixvalence}
Here we give tables for the bounds obtained from the valence formula. We then list how many coefficients have been computed with a computer in order to obtain the claim. 
The calculations were done using GP/Pari and run in parallel on 6 cores split between on one desktop (Dell Inspiron, i5 processor) and one laptop (Samsung NP900X3L, i5-6200U processor); the GP code may be found on the second author's website (see \url{https://hkumath.hku.hk/~bkane/papers/PeterssonQF/polygonal-Petersson.gp}). Most of the calculations took approximately 6 real-time hours (i.e., 36 core-hours) per $\bm{a}$, while those where the bound from the valence formula was larger than $10^8$ required a longer calculation of approximately  2-3 real-time days (i.e., approximately 15 core-days) each. 
\begin{center}
\begin{tabular}{|c|c|c|c||c|c|c|c|}
\hline
\small{$\bm{a}$}&\small{$N$}&\small{subgroup} &\small{coeff.}&\small{$\bm{a}$}&\small{$N$}&\small{subgroup} &\small{coeff.}\\
\hline
\hline
\multirow{2}{*}{\small{$\vphantom{T^{T^T}}(1\ 1\ 1)^T$}}&\small{ $N\nmid 6$}& \small $\Gamma_{256,16}$ & \small $384$&
\multirow{2}{*}{\small $(1\ 1\ 12)^T$} &\small $N\neq 6$ &\small  $\Gamma_{12288,64}$ &\small $98\,304$
\\[1pt]
\cline{2-4}\cline{6-8}
&\small $N\mid 6$&\small $ \Gamma_{576,24}$ & \small $1\,152$&
&\small $N=6$  &\small $\Gamma_{5184,72}$&\small $31\,104$
\\[1pt]
\cline{1-4}\cline{5-8}
\multirow{2}{*}{\small $(1\ 1\ 2)^T$}& \small $N\neq 12$ & \small $\Gamma_{1024,32}$& \small $3\,072$&
\multirow{2}{*}{\small $(1\ 1\ 21)^T$} &\small $N\nmid 14$&  \small $\Gamma_{36288,72}$ &\small $248\,832$
\\[1pt]
\cline{2-4}\cline{6-8}
& \small $N=12$ & \small $\Gamma_{1152,24}$ & \small $2\,304$ &
&\small $N\mid 14$  & \small $\Gamma_{460992,392}$&\small $22\,127\,616$
\\[1pt]
\cline{1-4}\cline{5-8}
\small{$\vphantom{T^{T^T}}(1\ 1\ 3)^T$}&\small  all&\small  $\Gamma_{20736,72}$ &\small  $124\,416$ &
\multirow{2}{*}{\small $(1\ 1\ 24)^T$} &\small $N\neq 3$ &\small $\Gamma_{49152,128}$ &\small $786\,432$
\\[1pt]
\cline{1-4}\cline{6-8}
\multirow{2}{*}{\small $(1\ 1\ 4)^T$} &\small $N \neq  12$&\small  $\Gamma_{4096,64}$ &\small $24\,576$&
&\small $N=3$ & \small $\Gamma_{2592,9}$&\small $3\,888$
\\[1pt]
\cline{2-4}\cline{5-8}
&\small $N=12$  &\small $\Gamma_{2304,48}$&\small $9\,216$&
\small\multirow{2}{*}{$(1\ 2\ 2)^T$} &\small $N\neq 6$ &\small  $\Gamma_{1024,32}$ &\small $3\,072$
\\[1pt]
\cline{1-4}\cline{6-8}
\small $\vphantom{T^{T^T}}(1\ 1\ 5)^T$ &\small all &\small  $\Gamma_{40000,200}$ &\small $720\,000$&
\small&\small $N=6$ &\small  $\Gamma_{576,24}$&\small $1\,152$
\\[1pt]
\cline{1-4}\cline{5-8}
\multirow{2}{*}{\small $(1\ 1\ 6)^T$} &\small $N\neq 8$ &\small  $\Gamma_{41472,144}$ &\small $497\,664$&
\small$\vphantom{T^{T^T}}(1\ 2\ 3)^T$ &\small all&\small $\Gamma_{20736,144}$ &\small $248\,832$
\\[1pt]
\cline{2-4}\cline{5-8}
&\small  $N=8$&\small $\Gamma_{3072,32}$&\small $12\,288$ &
\small$\vphantom{T^{T^T}}(1\ 2\ 4)^T$ &\small all&\small $\Gamma_{4096,64}$ &\small $24\,756$
\\[1pt]
\cline{1-4}\cline{5-8}
\small $\vphantom{T^{T^T}}(1\ 1\ 8)^T$ &\small all &\small  $\Gamma_{16384,128}$ &\small $196\,608$&
\small \multirow{2}{*}{$(1\ 2\ 5)^T$} &\small $N\neq 8$&\small  $\Gamma_{160000,400}$ &\small $5\,760\,000$
\\[1pt]
\cline{1-4}\cline{6-8}
\small \multirow{2}{*}{$(1\ 1\ 9)^T$} &\small $N\neq 4$&\small  $\Gamma_{104976,324}$ &\small $2\,834\,352$&
\small&\small $N=8$ &\small  $\Gamma_{5120,32}$&\small $18\,432$
\\[1pt]
\cline{2-4}\cline{5-8}
&\small $N=4$  &\small $\Gamma_{576,8}$&\small $576$&
\small$\vphantom{T^{T^T}}(1\ 2\ 6)^T$ &\small all &\small $\Gamma_{20736,144}$ &\small $248\,832$
\\[1pt]
\cline{1-4}\cline{5-8}

\end{tabular}

\begin{tabular}{|c|c|c|c||c|c|c|c|}
\hline
\small{$\bm{a}$}&\small{$N$}&\small{subgroup} &\small{coeff.}&\small{$\bm{a}$}&\small{$N$}&\small{subgroup} &\small{coeff.}\\
\hline
\hline

\cline{1-4}\cline{5-8}
\small$\vphantom{T^{T^T}}(1\ 2\ 8)^T$ &\small all&\small  $\Gamma_{16384,128}$ &\small $196\,608$&
\small	$\vphantom{T^{T^T}}(1\ 8\ 16)^T$ &\small all&\small $\Gamma_{65536,256}$ &\small $1\,572\,864$
\\
\cline{1-4}\cline{5-8}
\small$\vphantom{T^{T^T}}(1\ 2\ 10)^T$ &\small $N\neq 4$&\small $\Gamma_{40000,200}$ &\small $720\,000$&
\small \multirow{2}{*}{$(1\ 8\ 40)^T$} &\small $N=4$&\small $\Gamma_{5120,32}$&\small $18\,432$
\\
\cline{2-4}\cline{6-8}

\small&\small $N=4$&\small  $\Gamma_{5120,32}$&\small $18\,432$&
	&\small $N=5$&\small $\Gamma_{20000,25}$ &\small $90\,000$
\\
\cline{1-4}\cline{5-8}

\small$\vphantom{T^{T^T}}(1\ 2\ 16)^T$ &\small all&\small  $\Gamma_{65536,256}$ &\small $1\,572\,864$&
\small	$\vphantom{T^{T^T}}(1\ 9\ 9)^T$ &\small all&\small $\Gamma_{46656,216}$ &\small $839\,808$
\\
\cline{1-4}\cline{5-8}

\small \multirow{2}{*}{$(1\ 3\ 3)^T$} &\small $N\neq 8$&\small  $\Gamma_{5184,72}$ &\small $31\,104$&
\small$\vphantom{T^{T^T}}(1\ 9\ 12)^T$ &\small all&\small $\Gamma_{746496,864}$ &\small $53\,747\,712$
\\
\cline{2-4}\cline{5-8}
&\small $N=8$&\small $\Gamma_{768,16}$&\small $1\,536$&
\small\multirow{3}{*}{$(1\ 9\ 21)^T$}&\small  $N\mid 14$&\small $\Gamma_{345744,196}$ &\small $8\,297\,856$
\\
\cline{1-4}\cline{6-8}
\small $\vphantom{T^{T^T}}(1\ 3\ 4)^T$ &\small all&\small  $\Gamma_{82944,288}$ &\small $1\,990\,656$&
\small&\small  $N\mid 6$&\small $\Gamma_{326592,216}$ &\small $6\,728\,464$
\\
\cline{1-4}\cline{6-8}
\small $\vphantom{T^{T^T}}(1\ 3\ 6)^T$ &\small all&\small $\Gamma_{20736,144}$ &\small $248\,832$&
\small&\small  $N=4$&\small $\Gamma_{4032,8}$ &\small $4\,608$
\\
\cline{1-4}\cline{5-8}
\small $\vphantom{T^{T^T}}(1\ 3\ 9)^T$ &\small all&\small $\Gamma_{46656,216}$ &\small $839\,808$&
\small\multirow{2}{*}{$(1\ 9\ 24)^T$} &\small $N\mid 8$&\small $\Gamma_{147456,128}$ &\small $2\,359\,296$
\\
\cline{1-4}\cline{6-8}
\small \multirow{3}{*}{$(1\ 3\ 10)^T$} &\small $N\mid 12$&\small $\Gamma_{103680,144}$ &\small $1\,492\,992$&
\small	&\small $N=3$ &\small $\Gamma_{23328,27}$&\small $104\,976$
\\
\cline{2-4}\cline{5-8}
&\small $N\mid 10$ &\small $\Gamma_{120000,200}$ &\small $2\,880\,000$&
\small\multirow{2}{*}{$(1\ 12\ 12)^T$} &\small $N\nmid 8$&\small $\Gamma_{82944,288}$ &\small $1\,990\,656$
\\
\cline{2-4}\cline{6-8}
&\small $N=15$ &\small $\Gamma_{405000,225}$&\small $14\,580\,000$&
\small	&\small $N\mid8$ &\small $\Gamma_{49152,128}$&\small $786\,432$
\\
\cline{1-4}\cline{5-8}
\small $\vphantom{T^{T^T}}(1\ 3\ 12)^T$ &\small all&\small $\Gamma_{82944,288}$ &\small $1\,990\,656$&
\small\multirow{2}{*}{$(1\ 16\ 24)^T$} &\small $N\mid 6$&\small $\Gamma_{331776,576}$ &\small $15\,925\,248$
\\
\cline{1-4}\cline{6-8}
\small $\vphantom{T^{T^T}}(1\ 3\ 18)^T$ &\small all &\small $\Gamma_{2304,16}$ &\small $4\,608$&
\small &\small $N\mid 8$&\small $\Gamma_{196608,256}$ &\small $6\,291\,456$
\\
\cline{1-4}\cline{5-8}
\small \multirow{3}{*}{$(1\ 3\ 30)^T$} &\small $N\mid 10$&\small $\Gamma_{120000,200}$ &\small $2\,880\,000$&
\small \multirow{2}{*}{$(1\ 21\ 21)^T$} &\small $N\nmid 14$&\small $\Gamma_{36288,72}$&\small $248\,832$
\\
\cline{2-4}\cline{6-8}
&\small $N\mid6$&\small $\Gamma_{25920,72}$&\small $186\,624$&  
\small&\small $N\mid 14$&\small $\Gamma_{460992,392}$ &\small $22\,127\,616$
\\
\cline{2-4}\cline{5-8}
&\small $N=4$&\small $\Gamma_{15360,32}$ &\small $73\,728$&
\small \multirow{2}{*}{$(1\ 24\ 24)^T$} &\small $N\nmid 8$&\small $\Gamma_{20736,144}$ &\small $248\,832$
\\
\cline{1-4}\cline{6-8}
\small \multirow{2}{*}{$(1\ 4\ 4)^T$} &\small $N\neq 8$&\small $\Gamma_{9216,96}$ &\small $73\,728$&\small  
\small	&\small $N\mid8$ &\small $\Gamma_{49152,128}$&\small $786\,432$
\\
\cline{2-4}\cline{5-8}
&\small $N=8$ &\small $\Gamma_{4096,64}$&\small $24\,576$&
\small	$\vphantom{T^{T^T}}(2\ 2\ 3)^T$ &\small all&\small $\Gamma_{82944,288}$ &\small $1\,990\,656$
\\
	\cline{1-4}\cline{5-8}
\small \multirow{2}{*}{$(1\ 4\ 6)^T$} &\small $N\nmid 8$&\small $\Gamma_{82944,288}$ &\small $1\,990\,656$&
\small\multirow{2}{*}{$(2\ 3\ 3)^T$} &\small $N\neq 8$&\small $\Gamma_{20736,144}$ &\small $248\,836$
\\
\cline{2-4}\cline{6-8}
	&\small $N\mid 8$ &\small $\Gamma_{12288,64}$&\small $98\,304$&
\small	&\small $N=8$&\small $\Gamma_{3072,32}$&\small $12\,288$
\\
	\cline{1-4}\cline{5-8}
\small $\vphantom{T^{T^T}}(1\ 4\ 8)^T$ &\small all &\small $\Gamma_{16384,128}$ &\small $196\,608$&
\small$\vphantom{T^{T^T}}(2\ 3\ 6)^T$ &\small all&\small $\Gamma_{20736,144}$ &\small $248\,832$
\\
	\cline{1-4}\cline{5-8}
\small \multirow{2}{*}{$(1\ 4\ 12)^T$} &\small $N\neq 8$&\small $\Gamma_{82944,288}$ &\small $1\,990\,656$&
\small\multirow{2}{*}{$(2\ 3\ 8)^T$} &\small $N\mid 12$ &\small $\Gamma_{82944,288}$ &\small $1\,990\,656$
\\
\cline{2-4}\cline{6-8}
	&\small $N=8$&\small  $\Gamma_{12288,64}$&\small $98\,304$&
\small&\small $N=8$&\small $\Gamma_{49152,128}$&\small $786\, 432$
\\
\cline{1-4}\cline{5-8}
\small\multirow{2}{*}{$(1\ 4\ 24)^T$} &\small $N\nmid 6$&\small $\Gamma_{49152,128}$ &\small $786\,432$&
\small	$\vphantom{T^{T^T}}(2\ 3\ 9)^T$ &\small all &\small $\Gamma_{2304,16}$ &\small $4\ 608$
\\
\cline{2-4}\cline{5-8}
&\small $N\mid 6$ &\small $\Gamma_{331776,576}$&\small $15\,925\,248$&
\small	$\vphantom{T^{T^T}}(2\ 3\ 12)^T$ &\small all&\small $\Gamma_{331776,576}$ &\small $15\,925\,248$
\\
\cline{1-4}\cline{5-8}
\small	$\vphantom{T^{T^T}}(1\ 5\ 5)^T$ &\small all&\small $\Gamma_{40000,200}$ &\small $720\,000$&
\small \multirow{2}{*}{$ (2\ 3\ 18)^T$} &\small $N\neq 4$&\small $\Gamma_{46656,216}$ &\small $839\,808$
\\
\cline{1-4}\cline{6-8}
\small \multirow{2}{*}{$(1\ 5\ 8)^T$} &\small $N\mid 8$ &\small $\Gamma_{81920,128}$ &\small $1\,179\,648$&
\small	&\small $N=4$&\small $\Gamma_{9216,32}$&\small $36\,864$
\\
\cline{2-4}\cline{5-8}
 &\small $N=5$&\small $\Gamma_{20000,25}$ &\small $90\,000$&
\small\multirow{2}{*}{$(2\ 3\ 48)^T$} &\small $N\neq 3$&\small $\Gamma_{196608,256}$ &\small $6\,291\,456$
\\
\cline{1-4}\cline{6-8}
\small	$\vphantom{T^{T^T}}(1\ 5\ 10)^T$ &\small all&\small $\Gamma_{5120,32}$ &\small $18\,432$&
\small	&\small $N=3$ &\small $\Gamma_{5184,9}$&\small $7\,776$
\\
\cline{1-4}\cline{5-8}
\small	$\vphantom{T^{T^T}}(1\ 5\ 25)^T$ &\small all&\small $\Gamma_{1600,8}$ &\small $1\,440$&
\small\multirow{3}{*}{$(2\ 5\ 6)^T$} &\small $N\mid 20$&\small  $\Gamma_{480000,400}$ &\small $23\,040\,000$
\\
\cline{1-4}\cline{6-8}
\small	$\vphantom{T^{T^T}}(1\ 5\ 40)^T$ &\small all &\small $\Gamma_{81920,128}$ &\small $1\,179\,648$&
\small	&\small $N\mid 6$ &\small  $\Gamma_{103680,144}$ &\small $1\,492\,992$
\\
\cline{1-4}\cline{6-8}
\small \multirow{2}{*}{$(1\ 6\ 6)^T$} &\small $N\nmid 8$&\small $\Gamma_{5184,72}$ &\small $31\,104$&
\small	&\small $N=15$&\small $\Gamma_{40500,225}$&\small $14\,580\,000$
\\
\cline{2-4}\cline{5-8}
	&\small $N\mid8$&\small $\Gamma_{3072,32}$&\small $12\,288$&
\small	$\vphantom{T^{T^T}}(2\ 5\ 10)^T$ &\small all&\small $\Gamma_{5120,32}$ &\small $18\,432$
\\
\cline{1-4}\cline{5-8}
\small \multirow{2}{*}{$(1\ 6\ 9)^T$} &\small $N\neq 8$&\small $\Gamma_{186624,432}$ &\small $6\,718\,464$&
\small	$\vphantom{T^{T^T}}(2\ 5\ 15)^T$ &\small all&\small $\Gamma_{103680,144}$ &\small $1\,492\,992$
\\
\cline{2-4}\cline{6-8}
	&\small $N=8$ &\small $\Gamma_{9216,32}$&\small $36\,864$&
\small	$\vphantom{T^{T^T}}(2\ 6\ 9)^T$ &\small all&\small $\Gamma_{2304,16}$ &\small $4\,608$
\\
\cline{1-4}\cline{5-8}
\small \multirow{2}{*}{$(1\ 6\ 16)^T$} &\small $N\mid 12$&\small $\Gamma_{331776,576}$ &\small $15\,925\,248$ &
\small \multirow{2}{*}{$(2\ 6\ 15)^T$} &\small $N\nmid 20$ &\small $\Gamma_{103680,144}$&\small $1\,492\,992$
\\
\cline{2-4}\cline{6-8}
&\small $N=8$&\small $\Gamma_{196608,256}$ &\small $6\,291\,456$&
\small	&\small $N\mid 20$&\small $\Gamma_{480000,400}$ &\small $23\,040\,000$
\\
\cline{1-4}\cline{5-8}
\small	$\vphantom{T^{T^T}}(1\ 6\ 18)^T$ &\small all&\small $\Gamma_{2304,16}$ &\small $4\,608$&
\small{$(3\ 3\ 4)^T\vphantom{T^{T^T}}$} &\small all&\small $\Gamma_{331776,576}$ &\small $15\,925\,248$
\\
\cline{1-4}\cline{5-8}
\small \multirow{2}{*}{$(1\ 6\ 24)^T$} &\small $N\neq 8$&\small $\Gamma_{82944,288}$ &\small $1\,990\,656$&
\small\multirow{3}{*}{$(3\ 3\ 7)^T$} &\small $N=14$ &\small $\Gamma_{115248,196}$&\small $2\,765\,952$
\\
\cline{2-4}\cline{6-8}
	&\small $N=8$ &\small $\Gamma_{49152,128}$&\small $786\,432$&
\small	&\small $N\mid 21$&\small $\Gamma_{777924,441}$ &\small $56\,010\,528$
\\
\cline{1-4}\cline{6-8}
\small	$\vphantom{T^{T^T}}(1\ 8\ 8)^T$ &\small all&\small $\Gamma_{16384,128}$ &\small $196\,608$&
\small	&\small $N\in\{4,6\}$&\small $\Gamma_{36288,72}$ &\small $248\,832$
\\
\cline{1-4}\cline{5-8}
\end{tabular}
\end{center}
\begin{center}
	\begin{tabular}{|c|c|c|c||c|c|c|c|}
\hline
	$\bm{a}$&\small $N$ &\small  subgroup &\small  coeff.&\small 	$\bm{a}$&\small $N$ &\small  subgroup &\small  coeff.\\
\hline\hline

\small{$(3\ 3\ 8)^T\vphantom{T^{T^T}}$} &\small $N\mid 8$&\small $\Gamma_{49152,128}$ &\small $786\,432$&
\small{$(3\ 7\ 63)^T\vphantom{T^{T^T}}$} &\small $N=4$ &\small $\Gamma_{4032,8}$&\small $4\,608$
\\
\cline{2-4}\cline{5-8}
\small{$(3\ 3\ 8)^T\vphantom{T^{T^T}}$}&\small $N\in\{12,24\}$&\small $\Gamma_{1327104,1152}$ &\small $127\,401\,984$&
	$\vphantom{T^{T^T}}(3\ 8\ 8)^T$ &\small all &\small $\Gamma_{1327104,1152}$ &\small $127\,401\,984$
\\
\cline{1-4}\cline{5-8}
\small{	$(3\ 4\ 4)^T\vphantom{T^{T^T}}$} &\small all&\small  $\Gamma_{1327104,1152}$ &\small $127\,401\,984$&
\small{$(3\ 8\ 12)^T\vphantom{T^{T^T}}$}& \small all&\small $\Gamma_{1327104,1152}$ &\small $127\,401\,984$
\\
	\cline{1-4}\cline{5-8}
	\small{$(3\ 4\ 12)^T\vphantom{T^{T^T}}$} &\small all&\small $\Gamma_{1327104,1152}$ &\small $127\,401\,984$&
\small\multirow{2}{*}{$(3\ 8\ 48)^T$} &\small $N\nmid 8$ &\small $\Gamma_{331776,576}$&\small $15\,925\,248$
\\
\cline{1-4}\cline{6-8}
	$\vphantom{T^{T^T}}(3\ 4\ 36)^T$ &\small all&\small $\Gamma_{746496,864}$ &\small $53\,747\,712$&
&\small $N\mid 8$&\small $\Gamma_{196608,256}$ &\small $6\,291\,456$
\\
\cline{1-4}\cline{5-8}
\multirow{4}{*}{\small $(3\ 7\ 7)^T$} &\small $N\mid 21$&\small $\Gamma_{777924,441}$ &\small $56\,010\,528$&
\small\multirow{2}{*}{\small$(3\ 8\ 72)^T$} &\small $N=3$&\small $\Gamma_{23328,27}$ &\small $104\,976$
\\
	\cline{2-4}\cline{6-8}
	&\small $N=14$ &\small $\Gamma_{115248,169}$&\small $2\,765\,952$&
&\small $N=4$&\small $\Gamma_{9216,32}$&\small $36\,864$
\\
\cline{2-4}\cline{5-8} 
	&\small $N=4$ &\small $\Gamma_{1344,8}$ &\small $1\,536$&
\small{$(3\ 10\ 30)^T\vphantom{T^{T^T}}$} &\small all&\small $\Gamma_{103680,144}$ &\small $1\,492\,992$
\\
	\cline{2-4}\cline{5-8}
	&\small $N=6$ &\small $\Gamma_{9072,36}$ &\small $31\,104$&
\small{$(5\ 6\ 15)^T\vphantom{T^{T^T}}$} &\small all&\small $\Gamma_{103680,144}$ &\small $1\,492\,992$
\\
\cline{1-4}\cline{5-8}
\small{$(3\ 7\ 63)^T\vphantom{T^{T^T}}$} &\small $N\neq 4$&\small $\Gamma_{81648,108}$ &\small $839\,808$&
\small{$(5\ 8\ 40)^T\vphantom{T^{T^T}}$} &\small all&\small $\Gamma_{5120,32}$ &\small $18\,432$
\\
\cline{1-4}\cline{4-8}
\end{tabular}
\end{center}

\section{Bounds for the proofs of the lemmas}\label{sec:appendixvalencelemmas}
Here we give the bounds from the valence formula calculations in Section \ref{sec:lemmas}.
\begin{center}
		\begin{tabular}{|c||c|c|c|c|c|c|c|}
			\hline
			$\bm{a}$&\small  $\vphantom{T^{T^T}}(1\ 1\ 1)^T$ &\small  $(1\ 1\ 2)^T$ &\small  $(1\ 1\ 3)^T$ &\small  $(1\ 1\ 4)^T$ &\small  $(1\ 1\ 5)^T$ &\small  $(1\ 1\ 6)^T$ &\small  $(1\ 1\ 8)^T$ \\
			\hline

group &\small  $\Gamma_0(8)$ &\small  $\Gamma_0(16)$ &\small  $\Gamma_0(24)$ &\small  $\Gamma_0(16)$  &\small  $\Gamma_0(40)$  &\small  $\Gamma_0(24)$ &\small  $\Gamma_0(64)$\\
			\hline
			coeff. &\small  $2$ &\small  $3$  &\small  $6$ &\small  $3$ &\small  $9$ &\small  $6$ &\small  $12$\\
			\hline
			$\bm{a}$&\small  $\vphantom{T^{T^T}}(1\ 1\ 9)^T$ &\small  $(1\ 1\ 12)^T$ &\small  $(1\ 1\ 21)^T$ &\small  $(1\ 1\ 24)^T$ &\small  $(1\ 2\ 2)^T$ &\small  $(1\ 2\ 3)^T$ &\small  $(1\ 2\ 4)^T$ \\
			\hline

			group &\small  $\Gamma_0(72)$ &\small  $\Gamma_0(192)$ &\small  $\Gamma_0(1344)$ &\small  $\Gamma_0(192)$  &\small  $\Gamma_0(16)$  &\small  $\Gamma_0(48)$ &\small  $\Gamma_0(64)$\\
			\hline			
			coeff. &\small  $18$ &\small  $48$  &\small  $384$ &\small  $48$ &\small  $3$ &\small  $12$ &\small  $12$\\
			\hline
			$\bm{a}$&\small  $\vphantom{T^{T^T}}(1\ 2\ 5)^T$ &\small  $(1\ 2\ 6)^T$ &\small  $(1\ 2\ 8)^T$ &\small  $(1\ 2\ 10)^T$ &\small  $(1\ 2\ 16)^T$ &\small  $(1\ 3\ 3)^T$ &\small  $(1\ 3\ 4)^T$ \\
			\hline
			
			group &\small  $\Gamma_0(80)$ &\small  $\Gamma_0(48)$ &\small  $\Gamma_0(256)$ &\small  $\Gamma_0(80)$  &\small  $\Gamma_{1024,32}$  &\small  $\Gamma_0(24)$ &\small  $\Gamma_0(192)$\\
			\hline
			coeff. &\small  $18$ &\small  $12$  &\small  $384$ &\small  $18$ &\small  $3\,072$ &\small  $6$ &\small  $48$\\
			\hline
			$\bm{a}$&\small $(1\ 3\ 6)^T\vphantom{T^{T^T}}$ &\small  $(1\ 3\ 9)^T$ &\small  $(1\ 3\ 10)^T$ &\small  $(1\ 3\ 12)^T$ &\small  $(1\ 3\ 18)^T$ &\small  $(1\ 3\ 30)^T$ &\small $(1\ 4\ 4)^T$\\
			\hline
			
			group &\small  $\Gamma_0(48)$ &\small  $\Gamma_0(72)$ &\small  $\Gamma_0(120)$ &\small  $\Gamma_0(192)$  &\small  $\Gamma_0(144)$  &\small  $\Gamma_0(240)$ &\small  $\Gamma_0(16)$ \\
			\hline
			coeff. &\small  $12$ &\small  $18$  &\small  $36$ &\small  $48$ &\small  $36$ &\small  $72$&\small  $3$ \\
			\hline
	$\bm{a}$ &\small  $(1\ 4\ 6)^T$ &\small  $(1\ 4\ 8)^T$ &\small  $(1\ 4\ 12)^T$ &\small  $(1\ 4\ 24)^T$ &\small  $(1\ 5\ 5)^T$ &\small $(1\ 5\ 8)^T$ &\small  $(1\ 5\ 10)^T$ \\
			\hline
			
			group &\small  $\Gamma_0(192)$ &\small  $\Gamma_0(64)$ &\small  $\Gamma_0(192)$  &\small  $\Gamma_0(192)$  &\small  $\Gamma_0(40)$&\small  $\Gamma_0(320)$ &\small  $\Gamma_0(80)$\\
			\hline			
			coeff.  &\small  $48$  &\small  $12$ &\small  $48$ &\small  $48$ &\small  $9$&\small  $72$ &\small  $18$ \\
			\hline
			$\bm{a}$&\small  $(1\ 5\ 25)^T\vphantom{T^{T^T}}$ &\small  $(1\ 5\ 40)^T$ &\small  $(1\ 6\ 6)^T$ &\small  $(1\ 6\ 9)^T$ &\small $(1\ 6\ 16)^T$ &\small  $(1\ 6\ 18)^T$ &\small  $(1\ 6\ 24)^T$ \\
			\hline
			
			group  &\small  $\Gamma_{200,5}$ &\small  $\Gamma_0(320)$  &\small  $\Gamma_0(48)$  &\small  $\Gamma_0(144)$&\small  $\Gamma_{3072,32}$ &\small  $\Gamma_0(144)$ &\small  $\Gamma_{768,16}$ \\
			\hline
			coeff.  &\small  $180$ &\small  $72$ &\small  $12$ &\small  $36$&\small  $12\,288$ &\small  $36$  &\small  $1\,536$ \\
			\hline
			$\bm{a}$ &\small  $(1\ 8\ 8)^T\vphantom{T^{T^T}}$ &\small  $(1\ 8\ 16)^T$ &\small  $(1\ 8\ 40)^T$  &\small  $(1\ 9\ 9)^T$ &\small  $(1\ 9\ 12)^T$ &\small  $(1\ 9\ 21)^T$&\small  $(1\ 9\ 24)^T$ \\
			\hline

			group &\small  $\Gamma_{256,16}$&\small  $\Gamma_{1024,32}$  &\small  $\Gamma_{1280,16}$ &\small  $\Gamma_0(72)$ &\small  $\Gamma_0(576)$  &\small  $\Gamma_0(504)$  &\small  $\Gamma_0(576)$\\
			\hline
			coeff. &\small  $512$&\small  $3\,072$ &\small  $2\,304$  &\small  $18$ &\small  $144$ &\small  $144$  &\small  $144$ \\
			\hline
			$\bm{a}$ &\small  $(1\ 12\ 12)^T\vphantom{T^{T^T}}$ &\small  $(1\ 16\ 24)^T$ &\small  $(1\ 21\ 21)^T$ &\small  $(1\ 24\ 24)^T$ &\small  $(2\ 2\ 3)^T$ &\small $(2\ 3\ 3)^T$&\small  $(2\ 3\ 6)^T$\\
			\hline

			group  &\small  $\Gamma_0(48)$ &\small  $\Gamma_{3072,32}$ &\small  $\Gamma_0(168)$  &\small  $\Gamma_{768,16}$  &\small  $\Gamma_0(48)$ &\small  $\Gamma_0(48)$ &\small  $\Gamma_0(48)$\\
			\hline
			coeff. &\small $12$  &\small  $12\,288$ &\small  $48$ &\small  $1\,536$ &\small  $12$&\small  $12$ &\small  $12$\\
			\hline
			$\bm{a}$  &\small  $(2\ 3\ 8)^T\vphantom{T^{T^T}}$ &\small  $(2\ 3\ 9)^T$ &\small  $(2\ 3\ 12)^T$ &\small  $(2\ 3\ 18)^T$ &\small $(2\ 3\ 48)^T$ &\small  $(2\ 5\ 6)^T$ &\small  $(2\ 5\ 10)^T$ \\
			\hline
			
			group  &\small  $\Gamma_{768,16}$ &\small  $\Gamma_0(144)$  &\small  $\Gamma_0(192)$  &\small  $\Gamma_0(144)$&\small  $\Gamma_{3072,32}$ &\small  $\Gamma_0(240)$ &\small  $\Gamma_0(80)$ \\
			\hline
			coeff.   &\small  $1\,536$ &\small  $36$ &\small  $48$ &\small  $36$&\small  $12\,288$ &\small  $72$  &\small  $18$\\
			\hline
			
\end{tabular}
\end{center}		

\begin{center}			
\begin{tabular}{|c||c|c|c|c|c|c|c|}
			\hline
			$\bm{a}$&\small  $(2\ 5\ 15)^T\vphantom{T^{T^T}}$ &\small  $(2\ 6\ 9)^T$ &\small  $(2\ 6\ 15)^T$&\small $(3\ 3\ 4)^T$&\small $(3\ 3\ 7)^T$ &\small  $(3\ 3\ 8)^T$ &\small  $(3\ 4\ 4)^T$\\
			\hline
			group  &\small  $\Gamma_0(240)$  &\small  $\Gamma_0(576)$  &\small  $\Gamma_0(240)$ &\small $\Gamma_0(192)$&\small $\Gamma_0(168)$ &\small  $\Gamma_0(192)$ &\small  $\Gamma_0
(192)
$\\
			\hline			
			coeff.  &\small  $72$ &\small  $144$ &\small  $72$&\small   $48$ &\small $48$ &\small  $48$  &\small  $48$\\
			\hline
			$\bm{a}$ &\small  $(3\ 4\ 12)^T\vphantom{T^{T^T}}$ &\small  $(3\ 4\ 36)^T$ &\small $(3\ 7\ 7)^T$ &\small  $(3\ 7\ 63)^T$&\small $(3\ 8\ 8)^T$ &\small  $(3\ 8\ 12)^T$ &\small  $(3\ 8\ 48)^T$ \\
			\hline

			group  &\small  $\Gamma_0(192)$  &\small  $\Gamma_0(144)$ &\small  $\Gamma_0(168)$ &\small  $\Gamma_0(504)$ &\small $\Gamma_{768,16}$ &\small  $\Gamma_0(192)$ &\small  $\Gamma_{3072,32}$  \\
			\hline
			coeff.  &\small  $48$ &\small  $36$ &\small  $48$ &\small  $144$&\small  $1\,536$ &\small  $48$  &\small  $12\,288$  \\
			\hline
			$\bm{a}$&\small  $(3\ 8\ 72)^T\vphantom{T^{T^T}}$ &\small $(3\ 10\ 30)^T$ &\small  $(5\ 6\ 15)^T$ &\small  $(5\ 8\ 40)^T$&\multicolumn{3}{c}{}\\
			\cline{1-5}
			
			group &\small  $\Gamma_{2304,48}$ &\small  $\Gamma_0(240)$ &\small  $\Gamma_0(240)$ &\small  $\Gamma_{1280,16}$&\multicolumn{3}{c}{}\\
			\cline{1-5}
			coeff. &\small  $9\,216$&\small  $72$ &\small  $72$  &\small  $2\,304$&\multicolumn{3}{c}{}\\
\cline{1-5}
		\end{tabular}
\end{center}


\begin{thebibliography}{99}

\bibitem{ChanOh}W. Chan and B. Oh, \begin{it}Representations of integral quadratic polynomials,\end{it} in 
``Diophantine 
Methods, Lattices, and Arithmetic Theory of Quadratic Forms'', Contemp. Math. \textbf{587} (2013).
\bibitem{Ebel}I. Ebel, \begin{it}Analytische Bestimmung der Darstellungsanzahlen nat\"urlicher Zahlen durch spezielle tern\"are quadratische Formen mit Kongruenzbedingungen\end{it}, Math. Z. \textbf{64} (1956), 217-228.

\bibitem{HirzebruchZagier}F. Hirzebruch and D. Zagier, \begin{it}Intersection numbers of curves on Hilbert modular surfaces and modular forms of Nebentypus\end{it}, Invent. Math. \textbf{36} (1976), 57-113.
\bibitem{JacobsonMosunov}M. Jacobson, Jr. and A. Mosunov, \begin{it}Unconditional class group tabulation of imaginary quadratic fields to $|\Delta|<2^{40}$\end{it}, Math. Comp., to appear.
\bibitem{Jones}B. Jones, \begin{it} The arithmetic theory of quadratic forms\end{it}, Carus Math. Monographs, 1950.
\bibitem{OnoBook}K. Ono, \begin{it}The web of modularity: arithmetic of the coefficients of modular forms and $q$-series\end{it}, CMBS Regional Conference Series in Mathematics \textbf{102} (2004), American Mathematical Society, Providence, RI, USA.  
\bibitem{PeBer} H. Petersson, \begin{it}\"Uber die Berechnung der Skalarprodukte ganzer Modulformen\end{it}, Commentarii Mathematici Helvetici \textbf{22} (1949), 168-199.
\bibitem{Shimura} G. Shimura, \begin{it}On modular forms of half integral weight\end{it}, Ann. Math. \textbf{97} (1973), 440-481.
\bibitem{ShimuraCongruence}G. Shimura, \begin{it}Inhomogeneous quadratic forms and triangular numbers\end{it}, Amer. J. Math. \textbf{126} (2004), 191-214.
\bibitem{vanderBlij}F. van der Blij, \begin{it}On the theory of quadratic forms\end{it}, Ann. Math. \textbf{50} (1949), 875-883.
\bibitem{ZagierClassNum}D. Zagier, \begin{it}Nombres de classes et formes modulaires de poids $3/2$\end{it}, C.R. Acad. Sci. Paris (A) \textbf{281} (1975), 883-886.
\end{thebibliography}
\end{document}